%% file: Semibreak.tex
\documentclass[12pt]{amsart}
\usepackage{amscd,amsmath,amsthm,amssymb,mathtools}
\usepackage{mathrsfs}
\usepackage{comment}
\usepackage[all,cmtip]{xy}
\usepackage{color}
\usepackage{enumitem}
\usepackage[normalem]{ulem}
\usepackage{pgf,tikz}
\usetikzlibrary{arrows,matrix}

\definecolor{cadmiumgreen}{rgb}{0.0, 0.42, 0.24}
\usepackage[
colorlinks, citecolor=cadmiumgreen,
pagebackref,
pdfauthor={Andreas Gross, Farbod Shokrieh, Lilla T\'othm\'er\'esz}, 
pdftitle={Effective divisor classes on metric graphs}, 
pdfstartview ={FitV},
]{hyperref}

\usepackage[
alphabetic,
backrefs,
msc-links,
nobysame,
lite,
]{amsrefs}

\usepackage[left=3.5cm,top=3.5cm,right=3.5cm]{geometry}

\usepackage{microtype}
\usepackage{caption}
\usepackage[all,cmtip]{xy}
\usepackage{tikz, float} \usetikzlibrary {positioning}
\usetikzlibrary{calc,decorations.markings}
\usetikzlibrary{shapes,snakes}

\DeclareMathOperator\Error{Err}
\DeclareMathOperator\ME{ME}
\newcommand\diff{e}
\newcommand\minmax{\mathfrak{I}}
\DeclareMathOperator\spset{spset}

\def\supp{{\rm{supp}}}
\def\conv{{\rm{conv}}} 
\def\NN{{\mathbb N}}

\def\ZZ{{\mathbb Z}}
\def\RR{{\mathbb R}}
\def\CC{{\mathbb C}}

\def\UU{{\mathbf U}}

\def\B{{\mathcal B}}

\def\A{{\mathcal A}}

\def\opn#1#2{\def#1{\operatorname{#2}}} 
\opn\depth{depth} 
\opn\codim{codim}
\opn\ini{in} 
\opn\LM{LM}
\opn\LC{LC}
\opn\NF{NF}
\opn\Merge{Merge}
\opn\sgn{sgn}
\opn\div{div} 
\opn\Div{Div} 
\opn\Pic{Pic}
\opn\Jac{Jac}
\opn\Prin{Prin}
\opn\Del{Del}
\opn\op{op}
\opn\ends{ends}
\opn\indeg{indeg} 
\opn\outdeg{outdeg}
\opn\red{red}
\opn\Spec{Spec} 
\opn\Supp{Supp} 
\opn\Meas{Meas}
\opn\supp{supp} 
\opn\Ker{Ker} 
\opn\Coker{Coker} 
\opn\sign{sign}
\opn\Hom{Hom}
\opn\Tor{Tor} 
\opn\id{id}
\opn\cl{cl}
\opn\trace{trace}
\opn\Image{Image}
\opn\con{conv} 
\opn\relint{rel.int} 
\opn\vol{vol}
\opn\val{val}
\opn\Ber{Ber}
\opn\can{can}
\opn\syz{{\rm syz}}
\opn\spoly{{\rm spoly}}
\opn\LM{{\rm LM}}
\opn\lm{{\rm lm}}
\opn\lcm{{\rm lcm}} 
\opn\A{\mathcal A}
\opn\L{\mathrm{L}}
\opn\SB{\mathbb{SB}}
\opn\dist{dist}
\opn\size{size}
\opn\mult{mult}
\opn\pd{pd}
\opn\en{en}
\opn\PL{\mathrm{PL}}
\opn\H{\mathrm{H}}
\opn\V{\mathrm{V}}
\opn\B{\mathcal{B}}
\opn\Tr{\mathrm{Tr}}
\opn\Gr{\mathbf{Gr}}

\def\Implies{\ifmmode\Longrightarrow \else
        \unskip${}\Longrightarrow{}$\ignorespaces\fi}
\def\implies{\ifmmode\Rightarrow \else
        \unskip${}\Rightarrow{}$\ignorespaces\fi}
\def\iff{\ifmmode\Longleftrightarrow \else
        \unskip${}\Longleftrightarrow{}$\ignorespaces\fi}
\newtheorem{Theorem}{Theorem}[section]
\newtheorem{Lemma}[Theorem]{Lemma}

\newtheorem{Proposition}[Theorem]{Proposition}

\theoremstyle{remark}
\newtheorem{Remark}[Theorem]{Remark}

\theoremstyle{definition}

\newtheorem{Definition}[Theorem]{Definition}

\tikzstyle{Cwhite}=[scale = .8,circle, fill = white, minimum size=3mm] 
\tikzstyle{Cgray}=[scale = .4,circle, fill = gray, minimum size=3mm] 
\tikzstyle{Cblack2}=[scale = .4,circle, fill = black, minimum size=5mm] 
\tikzstyle{Cblack}=[scale = .7,circle, fill = black, minimum size=3mm]
\tikzstyle{C0}=[scale = .9,circle, fill = black!0, inner sep = 0pt, minimum size=3mm]
\tikzstyle{C1}=[scale = .7,circle, fill = black!0, inner sep = 0pt, minimum size=3mm]
\tikzstyle{Cred}=[scale = .4,circle, fill = red, minimum size=3mm]

\begin{document}

\title{Effective divisor classes on metric graphs}

\author{Andreas Gross}
\address{Imperial College London\\
Department of Mathematics\\
South Kensington Campus \\
London, SW7 2AZ \\
UK\\
}
\email{\href{mailto:a.gross@imperial.ac.uk}{a.gross@imperial.ac.uk}}

\author {Farbod Shokrieh}
\address{Cornell University\\
Ithaca, New York 14853-4201\\
USA}
\email{\href{mailto:farbod@math.cornell.edu}{farbod@math.cornell.edu}}

\author{Lilla T\'othm\'er\'esz}
\address{Cornell University\\
Ithaca, New York 14853-4201\\
USA\\
MTA-ELTE Egerv\'ary Research Group, P\'azm\'any P\'eter s\'et\'any 1/C, Budapest, Hungary}
\email{\href{mailto:tmlilla@math.cornell.edu}{tmlilla@math.cornell.edu}}

\subjclass[2010]{\href{https://mathscinet.ams.org/msc/msc2010.html?t=14T05}{14T05}, \href{https://mathscinet.ams.org/msc/msc2010.html?t=05C25}{05C25}}

\date{June 29, 2018}

\begin{abstract}
We introduce the notion of semibreak divisors on metric graphs (tropical curves) and prove that every effective divisor class (of degree at most the genus) has a semibreak divisor representative. This appropriately generalizes the notion of break divisors (in degree equal to genus). Our method of proof is new, even for the special case of break divisors. We provide an algorithm to efficiently compute such semibreak representatives.
Semibreak divisors provide the tool to establish some basic properties of effective loci inside Picard groups of metric graphs. We prove that effective loci are pure-dimensional polyhedral sets. We also prove that a `generic' divisor class (in degree at most the genus) has rank zero, and that the Abel-Jacobi map is `birational' onto its image. These are analogues of classical results for Riemann surfaces.
\end{abstract}

\maketitle

\setcounter{tocdepth}{1}
\tableofcontents

\section{Introduction}
Metric graphs, in many respects, are tropical (or non-Archimedean) analogues of Riemann surfaces. For example, there is a well-behaved theory of divisors and Jacobians for metric graphs (see e.g. \cite{BN07, MZ, GK, ABKS}). There is also an interesting interaction between the theories of divisors on metric graphs and on algebraic curves, with numerous applications in algebraic geometry (see e.g. \cite{baker, CDPR, JP1, JP2, CJP, BR, BPR, BJ}). The purpose of this work is to study {\em tropical effective loci} and establish some of their basic properties.

Let $\Gamma$ be a compact {\em metric graph} ({\em abstract tropical curve}) of genus $g$. Fix an integer $0 \leq d \leq g$. There is a canonical (Abel-Jacobi) map $S^{(d)} \colon \Div_+^d(\Gamma) \rightarrow \Pic^d(\Gamma)$ taking an effective divisor $D$ of degree $d$ on $\Gamma$ to its linear equivalence class $[D]$. The image of this map, denoted by $W_d$, is the {\em locus of effective divisor classes}. In the language of chip-firing games on metric graphs, this is the collection of chip configuration classes (up to chip-firing moves) which are `winnable'. 

We provide `nice' representatives for equivalence classes $[D] \in W_d$. In the case $d=g$, this is done by Mikhalkin and Zharkov in \cite{MZ} using the theory of tropical theta functions. They introduce the notion of `break divisors', and prove that every $[D] \in W_g = \Pic^g(\Gamma)$ has a unique break divisor representative. The notion of break divisors is further studied in \cite{ABKS} from a more combinatorial point of view related to orientations on graphs. 

A break divisor can be described as follows: pick $g$ disjoint open edge segments in $\Gamma$ so that, if we remove them from $\Gamma$, the remaining space becomes contractible (see the gray edges in Figure~\ref{fig:semibreak}). A break divisor is a divisor obtained by picking one point from {\em the closure} of each of these open edge segments (see Figure~\ref{fig:semibreak} (a), (b)). So a break divisor has degree equal to $g$ by construction. 

We define a {\em semibreak divisor} to be a divisor obtained from a break divisor after removing some points in its support. More precisely, a semibreak divisor is an effective divisor `dominated' by a break divisor (see Figure~\ref{fig:semibreak} (c), (d)). In particular, a break divisor is a semibreak divisor in degree $g$.

\vspace{-2mm}
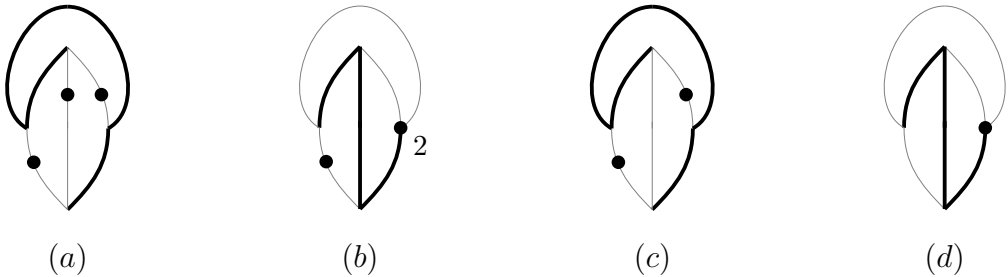
\begin{figure}[h!]
$$
\begin{xy}
(0,0)*+{
	\scalebox{.9}{$
	\begin{tikzpicture}
	\draw[black, gray] (0,1.2) to [out=-45,in=90] (.6,0);
	\draw[black, ultra thick] (.6,0) to [out=-90,in=45] (0,-1.2);
	\draw[black, ultra thick] (0,1.2) to [out=-135,in=90] (-.6,0);
	\draw[black, gray] (-.6,0) to [out=-90,in=135] (0,-1.2);
	\draw[black, gray] (0,1.2) -- (0,0);
	\draw[black, gray] (0,0.1) -- (0,-1.2);
	\draw[black, ultra thick] (.6,0) to [out=+30,in=0] (0,1.8);
	\draw[black, ultra thick] (0,1.8) to [out=180,in=150] (-.6,0);
	\fill[black] (0,.5) circle (.1);
	\fill[black] (.5,.5) circle (.1);
	\fill[black] (-.5,-.5) circle (.1);
	\end{tikzpicture}
	$}
};
(0,-20)*+{(a)};
\end{xy}
\ \ \ \ \ \ \ \ \ \ 
\begin{xy}
(0,0)*+{
	\scalebox{.9}{$
	\begin{tikzpicture}
	\draw[black, gray] (0,1.2) to [out=-45,in=90] (.6,0);
	\draw[black, ultra thick] (.6,0) to [out=-90,in=45] (0,-1.2);
	\draw[black, ultra thick] (0,1.2) to [out=-135,in=90] (-.6,0);
	\draw[black, gray] (-.6,0) to [out=-90,in=135] (0,-1.2);
	\draw[black, ultra thick] (0,1.2) -- (0,0);
	\draw[black, ultra thick] (0,0.1) -- (0,-1.2);
	\draw[black, gray] (.6,0) to [out=+30,in=0] (0,1.8);
	\draw[black, gray] (0,1.8) to [out=180,in=150] (-.6,0);
	\fill[black] (.6,0) circle (.1);
	\fill[black] (-.5,-.5) circle (.1);
	\end{tikzpicture}
	$}
};
(8,-5)*+{\mbox{{\smaller $2$}}};
(0,-20)*+{(b)};
\end{xy}
\ \ \ \ \ \ \ \ \ \ 
\begin{xy}
(0,0)*+{
	\scalebox{.9}{$
	\begin{tikzpicture}
	\draw[black, gray] (0,1.2) to [out=-45,in=90] (.6,0);
	\draw[black, ultra thick] (.6,0) to [out=-90,in=45] (0,-1.2);
	\draw[black, ultra thick] (0,1.2) to [out=-135,in=90] (-.6,0);
	\draw[black, gray] (-.6,0) to [out=-90,in=135] (0,-1.2);
	\draw[black, gray] (0,1.2) -- (0,0);
	\draw[black, gray] (0,0.1) -- (0,-1.2);
	\draw[black, ultra thick] (.6,0) to [out=+30,in=0] (0,1.8);
	\draw[black, ultra thick] (0,1.8) to [out=180,in=150] (-.6,0);
%	\fill[black] (0,.5) circle (.1);
	\fill[black] (.5,.5) circle (.1);
	\fill[black] (-.5,-.5) circle (.1);
	\end{tikzpicture}
	$}
};
(0,-20)*+{(c)};
\end{xy}
\ \ \ \ \ \ \ \ \ \ 
\begin{xy}
(0,0)*+{
	\scalebox{.9}{$
	\begin{tikzpicture}
	\draw[black, gray] (0,1.2) to [out=-45,in=90] (.6,0);
	\draw[black, ultra thick] (.6,0) to [out=-90,in=45] (0,-1.2);
	\draw[black, ultra thick] (0,1.2) to [out=-135,in=90] (-.6,0);
	\draw[black, gray] (-.6,0) to [out=-90,in=135] (0,-1.2);
	\draw[black, ultra thick] (0,1.2) -- (0,0);
	\draw[black, ultra thick] (0,0.1) -- (0,-1.2);
	\draw[black, gray] (.6,0) to [out=+30,in=0] (0,1.8);
	\draw[black, gray] (0,1.8) to [out=180,in=150] (-.6,0);
	\fill[black] (.6,0) circle (.1);
%	\fill[black] (-.5,-.5) circle (.1);
	\end{tikzpicture}
	$}
};
(0,-20)*+{(d)};
\end{xy}
\ \ \ \ \ \ \ \ \ \ 
\!\!\!\!\!\!\!\!\!\!\!\!\!\!\!\!\!\!\!
$$
\captionsetup{singlelinecheck=off}
\caption[]{Semibreak divisors on a metric graph of genus $3$.
\begin{itemize}
\item[(a)] a `generic' break divisors.
\item[(b)] a break divisor with some endpoints of open edge segments.
\item[(c)] a semibreak divisor in degree $2$ dominated by the break divisor in (a).
\item[(d)] a semibreak divisor in degree $1$ dominated by the break divisor in (b).
\end{itemize}
}
\label{fig:semibreak}
\end{figure}
\vspace{-2mm}
Most of our work is devoted to proving the following result.

\vspace{3mm}
\noindent{\bf Theorem A.} 
There exists a semibreak divisor in each $[D]\in W_d$.

\vspace{1mm}
See Theorem~\ref{thm:repr_with_below_break}.
\vspace{2mm}

Our techniques are new, even in the special case that $d=g$. We rely mostly on the geometry (e.g. a notion of `convexity') and algebraic topology (e.g. various Mayer-Vietoris sequences) of $\Gamma$ and its subspaces.
We avoid the use of tropical theta functions, orientations, semimodels, reduced divisors, etc. 

We start by a key result which states, roughly, that a given divisor is a break divisor if and only if a certain inequality holds for every `admissible' subset of $\Gamma$ (see Proposition~\ref{prop:break_char} for a precise statement). 
By a Mayer-Vietoris argument, this characterization relates break divisors to submodular functions (see \S\ref{sec:submod}). So, in its core, our approach resembles classical combinatorial proofs using submodularity, however, we exploit these ideas directly in a `continuous setting'. We find it remarkable that the `discrete theory' of submodular functions fits so naturally into our tropical setting.

Having this characterization, the naive strategy to prove Theorem A is straightforward: given an effective divisor $D$, we should first find a linearly equivalent divisor $D'$ that satisfies all the desired inequalities. We then would like to add points to $D'$ carefully in a way that all the desired inequalities are preserved. This process should eventually stop and output a break divisor that dominates a semibreak divisor linearly equivalent to $D$. It turns out that the construction of a suitable break divisor is more subtle than by simply adding points to $D'$. The process will rely on understanding certain canonical subsets $\minmax(E) \subsetneq \Gamma$ attached to effective divisors $E$. In a key result (Proposition~\ref{prop:comparing error after firing}) we will describe exactly how various invariants change as points are moved in relation to $\minmax(E)$.

\vspace{2mm}
We then turn our attention to the question of uniqueness.  

\vspace{3mm}
\noindent{\bf Theorem B.} 
There is a unique break divisor in each $[D] \in W_g$. If $d< g$ there can be distinct semibreak divisors in $[D] \in W_d$. If $\Gamma$ is $(d+1)$-edge connected then there is a unique semibreak divisor in each $[D] \in W_d$.

\vspace{1mm}
See Proposition~\ref{prop:uniq}, Figure~\ref{fig:equivalent}, Proposition~\ref{prop:distinctsemibreak} for precise statements. 
\vspace{2mm}

As mentioned earlier, the first statement is already proved in \cite{MZ, ABKS}. Our proof indicates that the result can be thought of as a consequence of a `maximum principle' (Lemma~\ref{lem:maxprin}).

\vspace{2mm}

We will then prove that an `integral' version of Theorem A also holds. A finite unweighted graph $G$ may be thought of as a metric graph whose edges have length 1. We define an integral divisor to be a divisor supported on the vertices of $G$.

\vspace{3mm}
\noindent{\bf Theorem C.} 
Let $[D] \in W_d$ and assume $D$ is integral. Then every semibreak divisor in $[D]$ is also integral.

\vspace{1mm}
See Proposition~\ref{prop:integral} for a precise statement. 
\vspace{2mm}

The fact that there {\em exists} an integral semibreak divisor in $[D]$ is immediate from our method of proof of Theorem A. More work is needed to show that {\em all} semibreak divisors in the equivalence class are indeed integral.  We note that the `existence' part of Theorem C may be stated purely in terms of finite graphs (avoiding metric graphs). We also remark that one could modify our proof to give a purely combinatorial proof of the finite graph version of Theorem C. One could also give another combinatorial proof (of the `existence' part) by using the theory of partial orientations in \cite{backman} and results in \cite{ABKS} (see Remark~\ref{rmk:partialorien}).

\vspace{2mm}

We will show that everything is efficiently computable.

\vspace{3mm}
\noindent{\bf Theorem D.} 
Given $[D] \in W_d$, there is an efficient algorithm that computes a semibreak divisor $D' \in [D]$.

\vspace{1mm}
See Theorem~\ref{thm:compute}, and \S\ref{sec:compute} for the description of the algorithm. 
\vspace{2mm}

We will reduce our computational problem to the theory of {\em submodular optimization} as in \cite{Schrijver_submod, IFF_submod}. For break divisors on finite graphs an algorithm is presented in \cite[\S7]{backman} relating the computation to the max-flow min-cut problem in graph theory.

Finally we apply the theory of semibreak divisors to prove the following tropical (non-Archimedean) analogues of some classical results on Riemann surfaces.

\vspace{3mm}
\noindent{\bf Theorem E.} 
\begin{itemize}
\item[(a)] $W_d$ is a {\em purely} $d$-dimensional polyhedral subset of $\Pic^d(\Gamma)$.
\item[(b)] The tropical (Abel-Jacobi) map $S^{(d)} \colon \Div_+^d(\Gamma) \rightarrow \Pic^d(\Gamma)$ is `birational' onto its image.
\item[(c)] There exists an open dense subset $U_d \subseteq W_d$ such that $r(D) = 0$ whenever $[D]\in U_d$.
\end{itemize}

\vspace{1mm}
See Theorem~\ref{thm:birational} and Theorem~\ref{thm:generic rank} for precise statements. 
\vspace{2mm}

The analogous statements for Riemann surfaces essentially follow from simple linear algebraic facts applied to `Brill--Noether' matrices (see e.g. \cite[p.245]{GriffHarr}). The situation in tropical geometry is different. While the fact that $W_d$ is a $d$-dimensional polyhedral subset of $\Pic^d(\Gamma)$ is elementary and well-known, the \emph{pure}-dimensionality of $W_d$ is much more subtle (see Remark~\ref{rmk:tropvsag1} (i)). We remark that one could give a highly non-constructive proof of pure-dimensionality by appealing to Berkovich's theory of non-Archimedean analytic spaces (\cite{Berkovich}) and combining the results in \cite{BR} and \cite{Gubler} (see Remark~\ref{rmk:tropvsag1} (ii)). Finally, the statement analogous to part (c) for Riemann surfaces (see e.g. \cite[p.245]{GriffHarr}) is usually stated as $r(D) = 0$ for a generic {\em effective divisor} $D$. In algebraic geometry, this is equivalent to saying $r(D) = 0$ for a generic {\em effective divisor class} $[D]$. In tropical geometry these two statements are {\em not} equivalent and, in fact, the former statement is not true (see Remark~\ref{rmk:tropvsag2}).

\medskip

\noindent{\em Structure of the paper.}
In \S\ref{sec:defs} we will review some basic definitions and set our notations and terminology.
In \S\ref{sec:semibreak} the notion of semibreak divisors is introduced. We will also state and prove a key result (Proposition~\ref{prop:break_char}) characterizing break divisors in terms of certain inequalities arising from the topology of `admissible' subsets of the metric graph.
In \S\ref{sec:submod} and \S\ref{sec:err} we study the functions and subsets related to the inequalities in Proposition~\ref{prop:break_char}. 
In \S\ref{sec:main} the existence of semibreak divisors (Theorem A) is proved. We will also discuss the uniqueness issues (Theorem B) and consider the integral version of semibreak divisors (Theorem C).
In \S\ref{sec:compute} we show how one can efficiently compute a semibreak divisor linearly equivalent to a given effective divisor (Theorem D).
In \S\ref{sec:generic} we apply the theory of semibreak divisors to prove basic generic properties of tropical effective loci (Theorem E).

\subsection*{Acknowledgments}
We would like to thank Matt Baker, Tam\'as Kir\'aly, Ye Luo, and Sam Payne for helpful conversations.
AG was supported by the ERC Starting Grant MOTZETA (project 306610) of the European Research Council (PI: Johannes Nicaise). LT was supported by NSF grant DMS-1455272 and by the Hungarian Scientific Research Fund - OTKA, K109240.

\section{Definitions and background} \label{sec:defs}

\subsection{Metric graphs and vertex sets}
\begin{Definition}
A {\em metric graph} (or an {\em abstract tropical curve}) is a compact connected metric space such that every point has a neighborhood isometric to a star-shaped set, endowed with the path metric.
\end{Definition}
By a {\em star-shaped} set of finite {\em valency} $n$ and {\em radius} $r$ we mean a set of the form
\[
S(n, r) = \{ z \in \CC \colon z=te^{k \frac{2\pi i}{n}} \text{for some } 0 \leq t < r \text{ and some } k \in \ZZ \} \, .
\]
In other words, a metric graph is a pair $(\Gamma,d)$ consisting of a compact connected topological graph $\Gamma$, together with an {\em inner metric} $d$. 

\medskip
{\itshape
As all of our results are easy to show for metric circles, we will always assume that $\Gamma$ is neither a point, nor a circle.
}
\medskip

\begin{Definition}
\begin{itemize}
\item[]
\item[(i)] The points of $\Gamma$ that have valency different from $2$ are called \emph{branch points} of $\Gamma$. 
\item[(ii)]A {\em vertex set} for $\Gamma$ is a finite set of points of $\Gamma$ containing all the branch points. 
\end{itemize}
\end{Definition}
We denote the minimal vertex set (i.e.\ the set of branch points) of $\Gamma$ by $V_\Gamma$. Note that $V_\Gamma$ is finite because $\Gamma$ is assumed to be compact. Also, $V_\Gamma$ is nonempty, because $\Gamma$ is not a circle. 

We denote the set of components of  $\Gamma \backslash V_\Gamma$ by $E_\Gamma$ and call its elements the \emph{open edges} of $\Gamma$.
By a \emph{closed edge} we mean the closure $\bar{e}$ of an open edge $e\in E_\Gamma$.
An open connected subset of an open edge of $\Gamma$ is called an \emph{open edge segment}, and the closure of such a segment is called a \emph{closed edge segment}.
If $e$ is a (closed or open) edge segment, the points in the topological boundary $\partial e$ of $e$ are called its \emph{endpoints}. Every edge has either $1$ or $2$ endpoints. 

Every finite combinatorial graph $G$ whose edges are weighted with positive real numbers naturally determines a metric graph $\Gamma_G$. A \emph{model} of a metric graph $\Gamma$ is a finite combinatorial weighted graph $G$, together with an isometry $\Gamma_G\xrightarrow{\phi} \Gamma$. Up to isomorphisms, the model $G$ is completely determined by the set $\phi(V(G))$, which is a vertex set for $\Gamma$, and conversely every vertex set determines a unique (up to isomorphism) model. If $G$ is the model of $\Gamma$ corresponding to the vertex set $V=V(G)$ (the isometry $\phi$ being implicit), then we call the elements of the set $E(G)$ of components of $\Gamma\backslash V$ the open edges of $G$. Of course, every open edge of $G$ is an open edge segment of $\Gamma$.

\begin{Remark}
Open edge segments of $\Gamma$ may be thought of as open edges of {\em some} model of $\Gamma$.
\end{Remark}

\begin{Definition} \label{def:conv} 
Let $S$ be a subset of a metric graph $\Gamma$.
The \emph{convex hull} of $S$, denoted by $\conv(S)$, is defined to be the union of $S$ and all closed edge segments whose endpoints are contained in $S$. 
\end{Definition}
It is not difficult to see that one can obtain $\Gamma \backslash \conv(S)$ by removing all connected components from $\Gamma \backslash S$ that are contained in some open edge of $\Gamma$.
A set $S\subseteq \Gamma$ is called \emph{convex}, if $\conv(S)=S$.

\begin{Definition}
We call a subset $S\subseteq \Gamma$ \emph{admissible} if it has only finitely many path-connected components. 
\end{Definition}
\begin{Remark}\label{rmk:admis}
\begin{itemize}
\item[]
\item[(i)] A subset $S\subseteq \Gamma$ is admissible if and only if there is a model $G$ of $\Gamma$ such that $S$ is a finite union of vertices and open edges of $G$. We say such a model $G$ is {\em compatible} with the admissible subset $S$. 
\item[(ii)] The collection of admissible sets is closed under finite Boolean combinations. 
\end{itemize}

\medskip

Let $S$ be a closed admissible subset of $\Gamma$ and let $p \in S$. For any sufficiently small star-shaped open neighborhood $B$ of $p$, the set $B\backslash S$ is a disjoint union of finitely many open edge segments, the number of which only depends on $p$.  We denote this number by $\val_S(p)$. Informally, this is the number of edges emanating from the admissible set $S$ at a point $p$. Clearly $\val_S(p) = 0$ if $ p \not\in \partial S$, where $\partial S$ denotes the topological boundary of $S$.

\end{Remark}
\begin{Definition} \label{def:genus}
Let $S$ be an admissible subset of a metric graph $\Gamma$.
\begin{itemize}
\item[(i)] The \emph{arithmetic genus} of $S$ is defined as $p_a(S) = 1-\chi(S)$, where $\chi(S)=\dim H^0(S;\RR)-\dim H^1(S;\RR)$ is the usual topological Euler characteristic of $S$. Here, $H^i(S;\RR)$ is the $i$-th singular cohomology group of $S$ with real coefficients.
\item[(ii)] The \emph{genus contribution} of $S$ is defined as $\psi(S)=p_a(\Gamma)-p_a(\Gamma \backslash S)$. 
\end{itemize}
\end{Definition}

\begin{Remark}\label{rmk:topology}
\begin{itemize}
\item[]
\item[(i)] A graph theorist might want to think of $H^1(S;\RR)$ as the vector space of $\RR$-valued {\em flows} on $S$. Moreover, $\dim H^0(S;\RR)$ is the number of connected  components of $S$. For a fine enough model $G$, the arithmetic genus $p_a(S)$ is equal to the number of closed edges corresponding to $G$ lying entirely inside $S$, minus the number of vertices of $G$ lying inside $S$, plus $1$. 
\item[(ii)] Recall that the geometric (or topological) genus of $S$ is defined as $p_g(S)= \dim H^1(S;\RR)$. Clearly, $p_g(S) = p_a(S)$ if and only if $S$ is connected. If $S$ is connected, we will refer to $p_g(S) = p_a(S)$ as the \emph{genus} of $S$.
\item[(iii)] If $S_1$ and $S_2$ are two open subsets of $\Gamma$, then the Euler characteristic obeys a version of the inclusion--exclusion principle:
\[
\chi(S_1\cup S_2) = \chi(S_1)+\chi(S_2) - \chi(S_1\cap S_2) \,.
\]
This follows from the Mayer-Vietoris sequence:
\[
\begin{aligned}
0 \!
&\rightarrow \!H^0(S_1\cup S_2, \RR)\!\!\!\!
&\rightarrow \! H^0(S_1, \RR) \oplus H^0(S_2, \RR) \!
&\rightarrow \! H^0(S_1\cap S_2, \RR) \! \rightarrow
 \\
&\rightarrow \!H^1(S_1\cup S_2, \RR) \!\!\!\!
&\rightarrow \! H^1(S_1, \RR) \oplus H^1(S_2, \RR) \!
&\rightarrow \! H^1(S_1\cap S_2, \RR) \! \rightarrow \! 0 \,.
\end{aligned}
\]
\end{itemize}
\end{Remark}

\subsection{Divisor theory on metric graphs}

Let $\Div(\Gamma)$ denote the free abelian group generated by the points of $\Gamma$. Denoting the generator corresponding to $p\in \Gamma$ by $(p)$, an element of $\Div(\Gamma)$, called a {\em divisor} on $\Gamma$, can be uniquely represented as 
\[D = \sum_{p \in \Gamma} a_p (p) \, ,\] 
where $a_p \in \ZZ$ and all but finitely many of the $a_p$ are zero. It is convenient to denote the coefficient $a_p$ in $D$ by $D(p)$. The {\em support} of $D$ is $\supp(D) = \{ p \in \Gamma \colon D(p) \ne 0\}$. A divisor $D \in \Div(\Gamma)$ is called \emph{effective} if $D(p)\geq 0$ for all $p\in \Gamma$.
For $D, E \in \Div(\Gamma)$, we write $E \leq D$ if $D-E$ is effective. The \emph{degree} of a divisor $D$ on $\Gamma$ is defined as $\deg(D) = \sum_{p \in \Gamma}D(p)$. More generally, the degree of a divisor $D$ on an admissible subset $S$ is 
 \[\deg(D|_S)=\sum_{p\in S}D(p) \, .\] 
 The set of divisors of a given degree $d$ is denoted by $\Div^d(\Gamma)$. The set of effective divisors of a given degree $d$ is denoted by $\Div_+^d(\Gamma)$.

Let $R(\Gamma)$ be the group of {\em continuous piecewise affine functions with integer slopes}. These are continuous functions $\phi\colon \Gamma\to \RR$ such that for every isometric map $\gamma\colon [0,\epsilon]\to \Gamma$, the pullback $\phi\circ \gamma$ is piecewise-linear with integral slopes in the usual sense. They are the tropical analogues of meromorphic functions on Riemann surfaces (\cite{MZ}). Note that such a function $\phi$ can only change its slope finitely many times on each closed edge.

Let 
\[\div \colon R(\Gamma) \rightarrow \Div(\Gamma)\] 
denote the Laplacian operator in the sense of distributions; for $\phi \in R(\Gamma)$, we have 
\[\div(\phi) =  \sum_{p \in \Gamma} {\sigma_p(\phi) (p)} \, ,\] 
where $\sigma_p(\phi)$ is the sum of incoming slopes of $\phi$ at $p$. It is easy to check that the group of {\em principal divisors} $\Prin(\Gamma) = \div(R(\Gamma))$ is contained in $\Div^0(\Gamma)$.

Two divisors $D_1$ and $D_2$ are called \emph{linearly equivalent}, written $D_1\sim D_2$, if there exists $\phi \in R(\Gamma)$ such that $D_1 - D_2 = \div(\phi)$. It is immediate that $\sim$ defines an equivalence relation. We denote the equivalence class of a divisor $D$ by $[D]$. The {\em complete linear system} $|D|$ is the set of all effective divisors linearly equivalent to $D$.

\begin{Remark} \label{rmk:chipfiring}
Given an effective divisor $D$, it is useful to think of $D(p)$ as
the number of {\em chips} placed at the point $p \in \Gamma$. 
For an admissible subset $S$ of $\Gamma$ and (sufficiently small) $\epsilon>0$, the rational function 
\[
\phi_{S,\epsilon}\colon\Gamma\to\RR,\;\;x\mapsto \min\{\epsilon, d(x,S)\} \, ,
\]
where $d$ is the metric on $\Gamma$, has value $0$ on $S$ and $\epsilon$ outside an $\epsilon$-neighborhood of $S$, with slope $1$ in each outgoing direction from $S$. 
Replacing $D$ with $D+\div(\phi_{S,\epsilon})$ has the effect of moving a chip to distance $\epsilon$ along each outgoing direction from $S$. This is often called `firing' the subset $S$ to distance $\epsilon$. One can check that every element of $R(\Gamma)$ can be written as a finite integer linear combination of functions of the form $\phi_{S,\epsilon}$. Therefore, one can describe linear equivalence of divisors on $\Gamma$ in terms of `chip-firing games'.
\end{Remark}

\section{Break and semibreak divisors} \label{sec:semibreak}
The notion of break divisors was introduced by Mikhalkin and Zharkov in \cite{MZ}, and further studied in \cite{ABKS}. Here, we introduce a natural generalization of this concept.

\begin{Definition}\label{def:genbr}
Let $\Gamma$ be a metric graph of genus $g$. 
\begin{itemize}
\item[(i)] A divisor $D$ on $\Gamma$ is called a  \emph{break divisor} if there exist $g$ disjoint open edge segments $e_1,\ldots, e_g$ and points $p_i\in \overline e_i$ such that $D=(p_1) + \cdots + (p_g)$ and $\Gamma \backslash \bigcup_{i=1}^g e_i$ is contractible.

\item[(ii)] A \emph{semibreak divisor} is an effective divisor $E$ such that $E\leq D$ for some break divisor $D$.
\end{itemize}
\end{Definition}
Note that a break divisor is also a semibreak divisor.

\begin{Remark}
\begin{itemize}
\item[]
\item[(i)] One can alternatively think of $\Gamma \backslash \bigcup_{i=1}^g e_i$ in Definition~\ref{def:genbr}(i) as a spanning tree of some model $G$ of the metric graph $\Gamma$.
\item[(ii)] The term `break' comes from the fact that a break divisor gives a recipe for fixing a fundamental domain in the maximal abelian covering of the metric graph $\Gamma$.
\item[(iii)] It is proved in \cite{MZ} and \cite{ABKS} that any equivalence class in degree $g$ has a unique break divisor representative. The uniqueness in this statement is relatively straightforward, and much of the work goes into proving the existence. In both \cite{MZ} and \cite{ABKS} the statement is proved by a careful study of the notion of `orientations'. 
One of our main results is that every effective divisor class of degree at most $g$ contains a semibreak divisor. Our approach is more topological in nature and, in particular, we completely avoid the use of orientations. 
\end{itemize}
\end{Remark}

We have the following useful characterization of break divisors. In the context of finite graphs, related statements can be found in the literature (see for example \cite[Proposition~4.11]{ABKS} and \cite[Theorem 3.4]{hiperTutte}). Some of the ideas in our proof below are adapted from the proof of \cite[Theorem 3.4]{hiperTutte}.

\begin{Proposition}
	\label{prop:break_char}
	Let $\Gamma$ be a metric graph of genus $g$, and let $D \in \Div^g(\Gamma)$. The following are equivalent:
\begin{itemize}
\item[(i)] $D$ is a break divisor.
\item[(ii)] $\deg(D|_S)\geq p_a(S)$ for all open admissible subsets $\emptyset \neq S\subseteq \Gamma$. 
\item[(iii)] $\deg(D|_S)\leq \psi(S)$ for all closed admissible subsets  $S\subsetneq \Gamma$.
\end{itemize}
\end{Proposition}
\begin{proof}
(ii)$\iff$(iii) follows immediately from exchanging $S$ with $\Gamma \backslash S$ and the definition of the genus contribution function $\psi$ (see Definition~\ref{def:genus} (ii)).

\medskip

(i)$\implies$(ii): Assume $D$ is a break divisor and let $e_1,\ldots, e_g$ be open edge segments in $\Gamma$ with $p_i\in \overline e_i$ such that $T=\Gamma\backslash \bigcup e_i$ is contractible and $D=\sum (p_i)$. Let $S$ be a nonempty open admissible subset of $\Gamma$. If $S\cap T=\emptyset$, then $S$ is a union of open edge segments and we have $p_a(S)\leq 0$ and there is nothing to prove. We may thus assume that $S\cap T\neq \emptyset$. 

Any sufficiently small open neighborhood of $T$ will intersect each $S\cap e_i$ in zero, one, or two connected components. Let $U$ be an open neighborhood  of $T$ such that $S\cap U\cap e_i$ has as few connected components as possible, for all $1\leq i\leq g$. 

\begin{itemize}
\item If $S\cap U\cap e_i$ has two connected components, then $\chi(S\cap U \cap e_i) = 2$ and $\chi(S\cap e_i) \geq 1$. 
\item If $S\cap U\cap e_i$ has one connected component, then $\chi(S\cap U \cap e_i) = 1$ and $\chi(S\cap e_i) \geq 1$.
\item If $S\cap U\cap e_i  = \emptyset $, then $\chi(S\cap U \cap e_i) = 0$ and $\chi(S\cap e_i) \geq 0$.
\end{itemize}
In any case, we have 
\begin{equation}\label{eq:eulerineq}
\chi(S\cap e_i) -\chi(S\cap U \cap e_i) +1 \geq 0 \, ,
\end{equation}
Equality in \eqref{eq:eulerineq} occurs if and only if $e_i\subseteq S$. 
\begin{itemize}
\item If $e_i\subseteq S$ and $p_i\notin S$, then $p_i$ is a boundary point of $e_i$ and at least one component of $S\cap U$ is contained entirely in $e_i$. 
\item As $U$ is a metric tree, every component of $S\cap U$ is a metric tree as well so that $\chi(S\cap U)$ equals the number of components of $S\cap U$. 
\end{itemize}
Together with the fact that $S\cap U$ has a component intersecting $T$ (recall that we assumed $S\cap T\neq \emptyset$), it follows that
\begin{equation}\label{eq:eulerineq2}
\chi(S\cap U) + \sum_{i\colon p_i\notin S} \big(\chi(S\cap e_i)-\chi(S\cap U\cap e_i)\big) \geq 1
\end{equation}

By the Mayer-Vietoris sequence (see Remark~\ref{rmk:topology}~(ii)), applied to the pair $(S\cap U,  S\cap (\bigcup_i  e_i))$, together with the fact that the Euler characteristic is additive with respect to disjoint unions of open sets, we have:

\[
\begin{aligned}
\chi(S) &= \chi(S\cap U) + \chi\bigg(S\cap \bigg(\bigcup_i  e_i\bigg)\bigg) - \chi\bigg(S\cap U \cap \bigg(\bigcup_i  e_i\bigg)\bigg)\\
&=\chi(S\cap U) + \sum_i \chi(S\cap e_i) - \sum_i \chi(S\cap U \cap e_i) \ .
\end{aligned}
\]
Therefore,
\[
\begin{aligned}
\chi(S)+\deg(D|_S) = 
\chi(S\cap U) &+ \sum_{i\colon p_i\notin S} \big(\chi(S\cap e_i)-\chi(S\cap U\cap e_i)\big) 
\\
&+\sum_{i\colon p_i\in S} \big(\chi(S\cap e_i)-\chi(S\cap U\cap e_i)+1\big)\, ,
\end{aligned}
\] 
which is at least $1$ by \eqref{eq:eulerineq} and \eqref{eq:eulerineq2}. Therefore, we have $\deg(D|_S)\geq p_a(S)$ as required. 
\medskip

(ii)$\implies$(i): Each point $p\in \Gamma$ has a neighborhood $B$ homeomorphic to a star-shaped set which is open and admissible. Since $\deg(D|_B) \geq p_a(B) = 0$, by taking $B$ to be sufficiently small so $\deg(D|_B) = D(p)$ we conclude that $D(p) \geq 0$. Therefore, $D$ must be effective. Let us denote $D=(p_1)+\dots +(p_g)$. 

Let $G$ be the model of $\Gamma$ corresponding to the vertex set $V_\Gamma\cup \supp(D)$. For each $p_i$ in the support of $D$, choose an open edge $e_i \in E(G)$ such that $p_i\in\bar{e}_i$. Note that for $i \ne j$ we might choose $e_i = e_j$. We define the {\em multiplicity} of the open edge $e \in E(G)$, denoted by $\mult(e)$, to be the number of indices $1\leq i\leq g$ with $e_i=e$.

Let $T=\Gamma \backslash \bigcup_{i=1}^g e_i$. Because we remove at most $g$ edges from the graph $\Gamma$ of genus $g$, we have $p_a(T)\geq 0$. Therefore, if the connected components of $T$ all have genus $0$, then $T$ must be connected. In this case, it is contractible and $D$ is a break divisor. We may thus assume that $T$ has a connected component of positive genus. We will show how to modify the choice of $\{e_1 , \ldots , e_g\}$ such that, eventually, $T$ becomes a spanning tree of $G$.

Let us denote by $m(T)$ the minimal number of edges of a component of $T$ of positive genus, and
let $C$ be a connected component of $T$ of positive genus with precisely $m(T)$ edges. After relabeling, we may assume that there exists $0\leq k\leq g$ such that for $1\leq i \leq g$ we have $p_i\in C$ if and only if $i\leq k$ ($k=0$ thus means that $\deg(D|_C)=0$).

\medskip
\noindent {\bf Claim.}
There exists $1\leq j\leq k$ such that {either} $\mult(e_j) \geq 2$, {or} $e_j$ connects $C$ to a different connected component of $T$. 
\medskip

\noindent {\em Proof of the Claim.} Suppose, to the contrary, that we have $\mult(e_i)=1$ and $\partial e_i \subseteq C$ for all $1\leq i\leq k$. Let $C'=C\cup \bigcup_{i=1}^k e_i$. Then $C'$ is closed and we have 
\[p_a(C')=k+p_a(C)  \geq k+1 \, .\] 
For any sufficiently small connected open neighborhood $V$ of $C'$ we then have
\[
\deg(D|_V)=\deg(D|_{C'})=\deg(D|_C)=k<k+1\leq p_a(C')=p_a(V)\, ,
\]
contradicting the assumption (ii).	

\medskip 

Let $1\leq j\leq k$ be as in the claim. We will substitute an edge $e'_j$ for $e_j$ as follows:
\begin{itemize}
\item [(1)] If $\mult(e_j)\geq 2$, then choose an arbitrary edge $e$ of $C$ incident to $p_j$ and set $e'_j=e$.
\item [(2)] If $\mult(e_j)=1$ and there exists an edge $e$ of $C$ incident to $p_j$ such that $C\backslash e$ is connected, we set $e'_j=e$.
\item [(3)] If $\mult(e_j)=1$, and every edge incident to $p_j$ disconnects $C$, we proceed as follows: for an edge $e$ of $C$ incident to $p_j$ we denote by $C_e$ the connected component of $C\backslash e$ that does not contain $p_j$. Then the sum $\sum_e p_g(C_g)$, where $e$ runs over all edges of $C$ incident to $p_j$, is equal to $p_g(C)$. In particular, this sum is positive, and hence there exists an edge $e$ of $C$ incident to $p_j$ such that $p_g(C_e)>0$. We set $e'_j=e$. 
\end{itemize}

Setting $e'_i= e_i$ for $i\neq j$, we have defined a new choice of edges. Denote $T'=\Gamma\backslash \bigcup e'_j$. If we are in case (1), then clearly $p_a(T')=p_a(T)-1$, whereas $p_a(T')=p_a(T)$ in cases (2) and (3). If we are in case (2), then the number of connected components of $T'$ is one less than the number of components of $T$. Therefore, we have $p_g(T')=p_g(T)-1$. In case (3), the number of components stays the same, so $p_g(T')=p_g(T)$. Finally, in case (3) the component $C_{e'_j}$ of $C\backslash e'_j$ is a component of $T'$ of positive genus. It clearly has less edges than $C$, and hence $m(T')< m(T)$. Summarizing this, we see that the triple $(p_a(T'), p_g(T'), m(T'))$ is strictly smaller than $(p_a(T), p_g(T), m(T))$ in the lexicographic order on $\NN^3$. As this is a well-order, we see that after modifying the choice $(e_i)$ of edges finitely many times as described above, we obtain a choice of edges whose complement has geometric genus $0$, finishing the proof.
\end{proof}

\section{(Sub)modularity of $\psi$} \label{sec:submod}

We will need the characterization of break divisors described in Proposition~\ref{prop:break_char} (iii) in terms of the genus contribution function $\psi$. Here, we record two important (sub)modularity properties of $\psi$.

\begin{Lemma}\label{l:f_is_submod}
Let $\Gamma$ be a metric graph. For any two closed admissible subsets  $S_1$ and $S_2$ of $\Gamma$ we have
\[
	\psi(S_1)+\psi(S_2)=\psi(S_1\cap S_2)+\psi(S_1\cup S_2) \, .
\]	
\end{Lemma}
\begin{proof}
By Definition~\ref{def:genus}, it suffices to show that
\[
\chi\left(\Gamma\backslash S_1\right) +\chi\left(\Gamma\backslash S_2\right)=\chi\left(\Gamma\backslash(S_1\cap S_2)\right)+\chi\left(\Gamma\backslash(S_1\cup S_2)\right) \, .
\]
This follows from the Mayer-Vietoris sequence, applied to the couple of open subsets $(\Gamma \backslash S_1 , \Gamma \backslash S_2)$ of $\Gamma\backslash(S_1\cap S_2)$ (see Remark~\ref{rmk:topology}~(ii)). 
\end{proof}

We will frequently need to pass to the convex hull. To express $\psi(\conv(S))$ in terms of $\psi(S)$, we need the following definition.

\begin{Definition} 
\begin{itemize}
\item[]
\item[(i)] Let $S$ be an admissible subset of $\Gamma$.  We define $\diff(S)$ to be the number of open edge segments in $\Gamma \backslash S$ whose endpoints are contained in $S$. 
\item[(ii)] For two admissible subsets $S_1$ and $S_2$ of $\Gamma$ we define $\diff(S_1,S_2) = \diff(S_1\cup S_2)$. 
\end{itemize}
\end{Definition}
\begin{Remark}
If $S_1$ and $S_2$ are admissible and {\em convex}, then $\diff(S_1,S_2)$ is precisely the number of open edge segments contained in $\Gamma \backslash (S_1\cup S_2)$ that have one endpoint in $S_1 \backslash S_2$ and one endpoint in $S_2 \backslash S_1$. 
\end{Remark}

\begin{Lemma}
\label{l:comparison with convex hull}
Let $\Gamma$ be a metric graph.
\begin{itemize}
\item[(a)] Let $S$ be a closed admissible subset of $\Gamma$. Then we have
\[
\psi(\conv(S))= \psi(S)-\diff(S) \, .
\]
\item[(b)] For any two closed admissible subsets  $S_1$ and $S_2$ of $\Gamma$ we have
\[
	\psi(S_1)+\psi(S_2)=\psi(S_1\cap S_2)+\psi(\conv(S_1\cup S_2))+\diff(S_1,S_2)
\]
\end{itemize}
\end{Lemma}

\begin{proof}
As noted after Definition~\ref{def:conv}, one obtains $\Gamma \backslash \conv(S)$ from $\Gamma \backslash S$ by removing from it all connected components that are contained in an open edge. These connected components are precisely the open edge segments in $\Gamma \backslash S$ whose endpoints are contained in $S$, so there are $\diff(S)$ of them. The topological Euler characteristic of each open edge segment is $1$. So by the additivity of the topological Euler characteristic on disjoint unions, we have $\chi(\Gamma \backslash \conv(S))=\chi(\Gamma \backslash S)-\diff(S)$. Part (a) now follows from the definition of $\psi$ (Definition~\ref{def:genus} (ii)).
Part (b) follows from part (a) and Lemma~\ref{l:f_is_submod}.
\end{proof}

\section{Error functions and  error sets} \label{sec:err}
 
 \begin{Definition}\label{def:errorsets}
 Let $\Gamma$ be a metric graph of genus $g$. Let $D \in \Div_+^d(\Gamma)$, with $0 \leq d \leq g$.
 \begin{itemize}
 \item[(i)] For a closed admissible subset $S \subseteq \Gamma$ we define the \emph{$D$-error of $S$} as the integer
 \[\Error(D,S) = \deg(D|_S)-\psi(S)\, .\] 
 \item[(ii)] The \emph{$D$-max error} is defined to be the integer 
\[
\ME(D)= \max\{ \Error(D,S) \colon S\subsetneq \Gamma \text{ closed and admissible}\} \ .
\]
\item[(iii)] A \emph{$D$-max error set} is a closed and admissible (not necessarily proper) subset $S \subseteq \Gamma$ with $\Error(D,S)=\ME(D)$. 
\end{itemize}
\end{Definition}

\begin{Remark}\label{rmk:Err}
In the context of finite graphs, the function $\chi$ studied in \cite[\S4.1]{ABKS} is closely related to a combinatorial analogue of our function $\Error$.
\end{Remark}

\begin{Lemma} \label{lem:Err}
 Let $\Gamma$ be a metric graph of genus $g$. Let $D \in \Div_+^d(\Gamma)$, with $0 \leq d \leq g$.
\begin{itemize}
\item[(a)] For any two closed admissible subsets $S_1$ and $S_2$ of $\Gamma$ we have:
\[
\Error(D,S_1)+\Error(D,S_2)=\Error(D,S_1\cap S_2)+ \Error(D,S_1\cup S_2)
\]
and
\[
\Error(D,S_1)+\Error(D,S_2)+\diff(S_1,S_2)\leq\Error(D,S_1\cap S_2)+ \Error(D,\conv(S_1\cup S_2)) \ .
\]
\item[(b)] We have $\ME(D) \geq 0$.
\item[(c)] $D$ is a break divisor if and only if $\deg(D)=g$ and $\ME(D)=0$.
\item[(d)]  $\Error(D,\Gamma)>\ME(D)$ if and only if $D$ is a break divisor.
\end{itemize}
\end{Lemma}
\begin{proof}
The equality in part (a) is a combination of Lemma~\ref{l:f_is_submod} and the obvious fact that 
\[
\deg(D|_{S_1})+\deg(D|_{S_2})=\deg(D|_{S_1\cap S_2})+\deg(D|_{S_1\cup S_2}) \, .
\]
To prove the inequality we use the same fact about degrees and combine it with Lemma~\ref{l:comparison with convex hull}~(b) and the fact that $\deg(D|_{\conv(S_1\cup S_2)})\geq \deg(D_{S_1\cup S_2})$ because $D$ is effective.

Part (b) follows from $\Error(D,\emptyset)=0$.

Part (c) follows Proposition~\ref{prop:break_char} and part (b).

For part (d), we first note that $\Error(D,\Gamma)= d-g+1$. If $D$ is a break divisor, we have $\Error(D,\Gamma) = 1$ but $\ME(D) = 0$ by part (c). Conversely, if $\Error(D,\Gamma)>\ME(D)$, then the nonnegativity of $\ME(D)$ implies that $d-g+1\geq 1$. This implies that $d=g$, and hence that $\ME(D)=0$. By part (c), we conclude that $D$ is a break divisor.
\end{proof}

\begin{Lemma}
\label{l:properties of max error sets}
 Let $\Gamma$ be a metric graph of genus $g$. Let $D \in \Div_+^d(\Gamma)$, with $0 \leq d \leq g$. Assume $\ME(D) >0$.
 \begin{itemize}
 \item[(a)] If $S$ is a $D$-max error set, then $S$ is convex.
  \item[(b)] If $S_1$ and $S_2$ are two $D$-max error sets, then $S_1\cup S_2$ and $S_1\cap S_2$ are also $D$-max error sets. Moreover $\diff(S_1,S_2)=0$.
  \end{itemize}
 \end{Lemma}

\begin{proof}

(a) If $S$ is not convex, then $\psi(\conv(S))<\psi(S)$ by Lemma \ref{l:comparison with convex hull}~(a). Since $S \subseteq \conv(S)$ we always have $\deg(D|_{\conv(S)})\geq \deg(D|_S)$. It follows that 
\[\Error(D,\conv(S))> \Error(D,S)=\ME(D) \, .\] 
If $\conv(S) \ne \Gamma$ this is a contradiction (see Definition~\ref{def:errorsets} (ii)). If $\conv(S) = \Gamma$ then 
\[ 1 \geq  \Error(D,\Gamma) > \ME(D) > 0\, ,\] 
which, again, is a contradiction.

(b) First, we observe that $\Error(D, S_1\cup S_2) \leq \ME(D)$. If $S_1 \cup S_2 \ne \Gamma$ this follows directly from Definition~\ref{def:errorsets}~(ii). If $S_1 \cup S_2 = \Gamma$, then it follows from Lemma~\ref{lem:Err}(d) and the fact that $D$ is not a break divisor by assumption.

Together with Definition~\ref{def:errorsets} (ii) and Lemma~\ref{lem:Err} (a), it	 follows that

\[
\begin{aligned}
2\ME(D) &\geq \Error(D,S_1\cap S_2)+ \Error(D, S_1\cup S_2) \\ 
&=\Error(D,S_1)+\Error(D,S_2) \\
&=2\ME(D) \, .
\end{aligned}
\]
Therefore, $ \Error(D,S_1\cap S_2)= \Error(D, S_1\cup S_2) = \ME(D)$, that is $S_1\cap S_2$ and $S_1\cup S_2$ are $D$-max error sets. By part (a), it follows that $S_1\cup S_2$ is convex and hence that $\diff(S_1,S_2)=\diff(S_1\cup S_2)=0$.
\end{proof}

\begin{Proposition}\label{prop:minmax}
Let $\Gamma$ be a metric graph of genus $g$. Let $D \in \Div_+^d(\Gamma)$, with $0 \leq d \leq g$. Then there exists a unique smallest (with respect to inclusion) $D$-max error set in $\Gamma$. If $G$ is the model corresponding to the vertex set $V(G)= V_\Gamma \cup \, \supp(D)$, then this smallest $D$-max error set is of the form
\begin{equation}
\label{eq:form of error set}
I\cup \bigcup_{e \in J} \bar{e}
\end{equation}
for some $I\subseteq V(G)$ and $J \subseteq E(G)$.
\end{Proposition}

\begin{proof}
Let $S$ be a $D$-max error set. Suppose that there exists a point in the boundary $\partial S$ of $S$ that is contained in an open edge $e \in E(G)$. Since $S$ is convex, $e\backslash S$ is either an open edge segment or a disjoint union of two open edge segments. So, by the Mayer-Vietoris sequence (see Remark~\ref{rmk:topology}~(ii)), we have
\[
\chi(\Gamma\backslash(S\backslash e))= \chi(\Gamma\backslash S)+ \chi(e)-\chi(e\backslash S) \leq \chi(\Gamma\backslash S)  
\]
and hence
$\psi(S\backslash e)\leq \psi(S)$. As we also have $\deg(D|_S)=\deg(D|_{S\backslash e})$ by definition of $G$, we see that $\Error(D,S\backslash e)\geq \Error(D,S)$ and hence that $S\backslash e$ is a $D$-max error set. This shows that every $D$-max error set contains a $D$-max error set of the form (\ref{eq:form of error set}). As there are only finitely many sets of this form, we see that every $D$-max error set contains an inclusion minimal $D$-max error set, and that all of these are of the form (\ref{eq:form of error set}). By Lemma~\ref{l:properties of max error sets}~(b), the intersection of all minimal $D$-max error sets is also a $D$-max error set, hence this is the unique smallest $D$-max error set.
\end{proof}

\begin{Definition}
We will denote the unique minimal $D$-max error subset of $\Gamma$ (as in Proposition~\ref{prop:minmax}) by $\minmax(D)$. 
\end{Definition}
\begin{Remark}\label{rmk:MEzero}
$\minmax(D)$ is always a proper subset of $\Gamma$, as the maximum error is taken for proper (admissible) subsets.
We have $\ME(D)=0$ if and only if $\minmax(D) = \emptyset$.
\end{Remark}

Our next goal is to prove a key result about $\minmax(D)$. We recall two standard notations. 
Let $(X, d)$ be a metric space.
\begin{itemize}  
\item For two nonempty subsets $A, B \subseteq X$, one defines
\[
\dist(A, B) = \inf\{d(x, y) \colon x \in A, y \in B\} \, .
\]
\item For a nonempty subset $A \subseteq X$, one defines its $\epsilon$-fattening by
\[
A_{\epsilon} = \bigcup_{x \in A} \{z \in X \colon d(z,x) \leq \epsilon \} \, .
\]
\end{itemize}

\begin{Proposition}
\label{prop:comparing error after firing}
 Let $\Gamma$ be a metric graph of genus $g$. Let $D \in \Div_+^d(\Gamma)$, with $0 \leq d \leq g$. Let $S = \minmax(D)$, and
 assume that $\ME(D) > 0$. 

 \begin{itemize}
\item[(a)] Let $\epsilon = \dist\left(S,V_{\Gamma} \backslash S\right)$. Let $D_1$ be the divisor obtained from $D$ by `firing' $S$ to distance $\epsilon$. In other words, $D_1 = D+\div(\phi_{S,\epsilon})$, where $\phi_{S,\epsilon}$ is as in Remark~\ref{rmk:chipfiring}. Then
\begin{itemize} 
\item[(i)] $D_1$ is effective. 
\item[(ii)] $\ME(D_1)\leq \ME(D)$. 
\item[(iii)] If $\ME(D_1)= \ME(D)$, then $S_\epsilon \subseteq \minmax(D_1)$.  
\end{itemize}
\vspace{2mm}
\item[(b)] Let $e^-\in \partial S$ and $e^+ \in V_\Gamma \backslash S$ be endpoints of an open edge segment $e \subseteq \Gamma \backslash S$. Let $D_2$ be the divisor obtained from $D$ by moving a chip from $e^-$ to $e^+$. In other words, $D_2=D+(e^+)-(e^-)$. Then 
\begin{itemize} 
\item[(i)]  $D_2$ is effective. 
\item[(ii)] $\ME(D_2)\leq \ME(D)$.
\item[(iii)] If $\ME(D_2) =  \ME(D)$, then $S \cup \{e^+\} \subseteq \minmax(D_2)$. 
\end{itemize}
 \end{itemize}
\end{Proposition}

\begin{proof}
We observe that neither can $S$ be empty, since $\ME(D)>0$, nor can it contain $V_\Gamma$, as this would imply $S=\conv(S)=\Gamma$, a contradiction (see Remark \ref{rmk:MEzero}).

\medskip
\noindent{\bf Claim 1.} For every $p\in S$ we have $\val_S(p)\leq D(p)$.

\medskip
\noindent{\em Proof of Claim 1.} If $p \not\in \partial S$ we have $\val_S(p)=0$ and there is nothing to prove. 
Let $p\in \partial S$. Let $B$ be a sufficiently small open neighborhood of $p$, isometric to a star-shaped set and not containing any point in $V_\Gamma\cup \,\supp(D)$ aside from $p$. 

By the minimality of $S$, we know $S\backslash B$ is not a $D$-max error set, so 
\begin{equation} \label{eq:1}
\Error(D,S) \geq \Error(D,S\backslash B) +1 \, .
\end{equation}

By the choice of $B$, we have 
\begin{equation} \label{eq:2}
\deg(D|_{S\backslash B}) = \deg(D|_S) -D(p)\, .
\end{equation}

Furthermore, we have
\[
\chi(\Gamma \backslash (S\backslash B))=\chi(\Gamma\backslash S)  +\chi (B) - \chi (B\backslash S) 
\]
by Mayer-Vietoris (Remark~\ref{rmk:topology}~(ii)), which equals $\chi(\Gamma\backslash S) +1 -\val_S(p)$ by the choice of $B$.
Therefore, 
\begin{equation} \label{eq:3}
\psi(S\backslash B)=\psi(S)+1-\val_S(p) \, .
\end{equation}
Combining \eqref{eq:1}, \eqref{eq:2}, and \eqref{eq:3}, we obtain
\[
\begin{aligned}
\Error(D,S)&\geq\Error(D,S\backslash B)+1\\
&=  \deg(D|_{S\backslash B}) - \psi(S\backslash B) +1 \\
&= \big(\deg(D|_S) -D(p)\big) - \big(\psi(S)+1-\val_S(p)\big)+1\\
&=  \Error(D,S)-D(p)+\val_S(p) \, ,
\end{aligned}
\]
from which we deduce that $\val_S(p) \leq D(p)$.

\medskip

(a) Note that $\epsilon = \dist\left(S,V_{\Gamma} \backslash S\right)$ is well-defined because $S$ and $V_{\Gamma} \backslash S$ are nonempty.
For any $p \in S$, we have $D_1(p) = D(p) - \val_S(p)$. So the effectiveness of $D_1$ follows directly from Claim~1.

Let $\UU$ denote the the set of all closed edge segments $e$ of length $\epsilon$ with endpoints $\partial e = \{e^-, e^+\}$ such that $e^- \in S$ and $\dist(e^+, S) = \epsilon$. `Firing' $S$ to distance $\epsilon$ has the effect of sending one chip from $e^-$ to $e^+$ for each $e \in \UU$. So, for any admissible subset $R \subseteq \Gamma$ we have
\begin{equation}\label{eq:degs}
\deg(D_1|_{R}) = \deg(D|_{R})- 
|\{e\in \UU\colon e^{-}\in R, e^+\notin R\}|
+|\{e\in \UU\colon e^{-}\not\in R, e^+\in R\}|  \, .
\end{equation}
Let $R = \minmax(D_1)$. By Lemma~\ref{lem:Err}~(a) and \eqref{eq:degs} we have:
\begin{equation}\label{eq:RS}
\begin{aligned}
&\Error(D,S\cap R)+\Error(D,\conv(S\cup R)) \geq
\Error(D,S)+\Error(D,R)+\diff(S,R)\\
 & \geq \Error(D,S)+\Error(D_1,R)+\diff(S,R)
 -|\{e\in \UU \colon e^-\notin R, e^+\in R\}| \ ,
 \end{aligned}
\end{equation}
with equality only if $\{e\in \UU\colon e^-\in R, e^+\notin R\}=\emptyset$. 

\medskip
\noindent{\bf Claim 2.} $|\{e\in \UU\colon e^-\notin R, e^+\in R \}|\leq \diff(S,R)$.

\medskip
\noindent{\em Proof of Claim 2.}
For any $e\in \UU$ with  $e^-\notin R$ and $e^+\in R$ we have $e\backslash \{e^-\}\nsubseteq R$ because $R$ is closed. Since $R$ is also convex (Lemma \ref{l:properties of max error sets}~(a)), $e$ contains a unique connected component of the complement of $S\cup R$ in its interior. This component is an open edge segment which has one endpoint in $S\backslash R$ and one endpoint in $R\backslash S$. It therefore contributes with $1$ to $\diff(S,R)$, proving the claim.

\medskip

By \eqref{eq:RS}, Claim 2, and the fact that $\Error(D,\conv(S\cup R)) \leq \ME(D)$ (see proof of Lemma~\ref{l:properties of max error sets}~(b)), we obtain: 
\[
\begin{aligned}
2\ME(D) &\geq\Error(D,S\cap R)+\Error(D,\conv(S\cup R)) \\
&\geq \Error(D,S)+ \Error(D_1,R) \\
&=\ME(D)+\ME(D_1) \ .
\end{aligned}
\]
It follows that $\ME(D_1)\leq \ME(D)$. In case of equality, $R$ is also a $D$-max error set and thus, by Lemma~\ref{l:properties of max error sets}~(b), so is $S\cap R$. Because $S = \minmax(D)$, we must have $S\subseteq R$. And since \eqref{eq:RS} is an equality, we have $\{e\in \UU\colon e^-\in R, e^+\notin R\} = \emptyset$. Therefore, $R$ must contain all points of $\Gamma$ that have distance $\epsilon$ to $S$. By the convexity of $R$ (Lemma \ref{l:properties of max error sets}~(a)), it follows that $R$ does in fact contain all points of distance at most $\epsilon$ to $S$.

\medskip

(b) It follows directly from Claim 1 that $D_2$ is also effective. 
Let $Q= \minmax(D_2)$. We have four cases: 

\begin{enumerate}

\item[{\em Case 1:}] $e^-,e^+\in Q$. We have 
\[\ME(D_2)=\Error(D_2,Q)=\Error(D,Q)\leq \ME(D)\, .\] 
In case of equality $Q$ is $D$-max error set and thus contains both $e^+$ and $S$.
\item[{\em Case 2:}] $e^-,e^+\notin Q$. As $Q$ does not contain $S$, it is not a $D$-max error set. Therefore,
\[\ME(D_2)=\Error(D_2,Q)=\Error(D,Q)<\ME(D) \, .\]
\item[{\em Case 3:}] $e^-\in Q$, $e^+\notin Q$. We have 
\[\ME(D_2)=\Error(D_2,Q)=\Error(D,Q)-1<\ME(D)\, .\]
\item[{\em Case 4:}] $e^-\notin Q$, $e^+\in Q$. We have $\ME(D_2) = \Error(D_2,Q)=\Error(D,Q)+1$. Since $Q$ is closed, the open edge segment $e$ is not contained in $S\cup Q$. It follows that $e$ contains a connected component of $\Gamma\backslash (S\cup Q)$. As such an edge segment is automatically an open edge segment with endpoints in $S\cup Q$, we have $\diff(S,Q)>0$ and hence 
\[
\begin{aligned}
\Error(D,S\cap Q)+\Error(D,\conv(S\cup Q))&\geq\Error(D,S)+\Error(D,Q)+e(S,Q)\\
&\geq\ME(D)+\ME(D_2)-1+1 \, .
\end{aligned}
\]

From this, and the fact $\Error(D,\conv(S\cup Q)) \leq \ME(D)$ (see proof of Lemma~\ref{l:properties of max error sets}~(b)), we conclude
\[
2\ME(D)\geq \Error(D,S\cap Q)+\Error(D,\conv(S\cup Q)) \geq \ME(D)+\ME(D_2) \ .
\]
It follows that $\ME(D_2)\leq \ME(D)$. Equality is not possible in this case because $Q$ does not contain $S$.
\end{enumerate}
\end{proof}

\section{Semibreak divisors in effective divisor classes} \label{sec:main}

\subsection{Existence of semibreak divisors}
We are now ready to prove our main theorem about the existence of semibreak divisors in effective classes.

\begin{Theorem}\label{thm:repr_with_below_break}
 Let $\Gamma$ be a metric graph of genus $g$. Let $D \in \Div^d(\Gamma)$, with $0 \leq d \leq g$, and assume $|D|\neq \emptyset$. Then $|D|$ contains a semibreak divisor.
 \end{Theorem}

\begin{proof}
The result is straightforward for metric circles so, as before, we will assume $\Gamma$ is not homeomorphic to a circle. 
Since $|D|\neq \emptyset$, we may also assume that $D$ is effective. 

Let $E$ be any effective divisor of degree $g-d$. If $\ME(D+E)=0$, we are done by Lemma~\ref{lem:Err}~(c). If $\ME(D+E)>0$, let $S= \minmax(D+E)$. We will show how to construct a pair $(D',E')$ of effective divisors such that
\begin{itemize}
\item[-] $D'\sim D$ and 
\item[-] either $\ME(D'+E')<\ME(D+E)$, or $\ME(D'+E')=\ME(D+E)$ and $ \minmax(D'+E')$ contains more branch points of $\Gamma$ than $S$. 
\end{itemize}
Since $\Gamma$ has only finitely many branch points, and a convex subset of $\Gamma$ containing all its branch points must be equal to $\Gamma$, this will prove the theorem. We consider two cases:
\begin{itemize}[leftmargin=*]
\item[(1)] $\supp(E)\cap \,  \partial S\neq \emptyset$. Then there exists an open edge segment $e \subseteq \Gamma \backslash S$ with endpoints $e^-\in \supp(E)\cap \,  \partial S$ and $e^+ \in V_\Gamma \backslash S$.
Set $D'=D$ and $E'=E-(e^-)+(e^+)$. Both $D'$ and $E'$ are effective by construction, and $D'\sim D$. By Proposition~\ref{prop:comparing error after firing} (b), we have $\ME(D'+E')\leq \ME(D+E)$, and if there is equality, then $\minmax\left(D'+E'\right)$ contains more branch points of $\Gamma$ than $S$.

	\item[(2)] $\supp(E)\cap \,  \partial S =\emptyset$. Let $\epsilon = \dist\left(S,V_{\Gamma} \backslash S\right)$, and consider the divisor obtained from $D+E$ by `firing' $S$ to distance $\epsilon$, i.e. $D+E+\div(\phi_{S,\epsilon})$ (see Remark~\ref{rmk:chipfiring}). By Proposition~\ref{prop:comparing error after firing}~(a), $D+E+\div(\phi_{S,\epsilon})$ is effective. Since $\partial S\cap \,  \supp(E)=\emptyset$, this implies that $D'=D+\div(\phi_{S,\epsilon})$ is also effective. Let $E'=E$.
Then, again by Proposition~\ref{prop:comparing error after firing} (a), we have $\ME(D'+E')\leq \ME(D+E)$, and if there is equality, then $\minmax\left(D'+E'\right)$ contains more branch points of $\Gamma$ than $S$.

\end{itemize}
\end{proof}

\subsection{Uniqueness issues}
The existence of semibreak divisors (Theorem~\ref{thm:repr_with_below_break}) is sufficient for the applications considered in \S\ref{sec:generic}. However, it is natural to wonder about uniqueness of such representatives. 

By a {\em cut} $C$ in a metric graph $\Gamma$ we mean a disjoint union of open edge segments that disconnects $\Gamma$. The {\em size} of a cut $C$, denote by $\size(C)$, is the number of connected components (maximal open edge segments) of $C$. If $S\subseteq \Gamma$ is a closed admissible subset then, for sufficiently small $\epsilon>0$, the set $\{ x \in \Gamma\backslash S : \dist(x,S)<\epsilon\}$ forms a cut. We say that such a cut is determined by $S$. The size of a cut determined by $S$ does not depend on any choices, and will be denoted by $c(S)$.

We start with two useful lemmas.
\begin{Lemma}
\label{lem:cut}
Let $\Gamma$ be a metric graph, and let $S\subseteq \Gamma$ be a closed admissible set. Then
\[
c(S) = \psi(S)-p_a(S)+1 \, .
\]

\end{Lemma}

\begin{proof}
Let $C$ be a cut determined by $S$. Because all components of $C$ are open edge segments we have 
\begin{equation}\label{eq:cut1}
\chi(C) = c(S) \,.
\end{equation} 
Let $S'=S\cup C$. Then $S$ is a deformation retract of the open and admissible set $S'$. Therefore
\begin{equation}\label{eq:cut2}
\chi(S)=\chi(S') \, .
\end{equation}
Applying the Mayer-Vietoris sequence (see Remark~\ref{rmk:topology}~(ii)) to the pair $(S',\Gamma\backslash S)$ yields
\begin{equation}\label{eq:cut3}
\chi(C)=\chi(S')+\chi(\Gamma\backslash S)-\chi(\Gamma) \, .
\end{equation}
The result follows from \eqref{eq:cut1}, \eqref{eq:cut2}, and \eqref{eq:cut3}.
\end{proof}

The following result is a generalized version of the `maximum principle' (see e.g. \cite[Lemma 3.7]{BS13})
\begin{Lemma}\label{lem:maxprin}
Let $\Gamma$ be a metric graph and $\phi\in R(\Gamma)$. Let $S$ be the subset of $\Gamma$ where $\phi$ attains its minimum. Then 
\begin{itemize}
\item[(a)] $S$ is closed and admissible. 
\item[(b)] For any $p \in S$ we have
$-\div(\phi)(p) \geq \val_S(p)$.
\end{itemize}
\end{Lemma}
\begin{proof}
Part (a) follows from the fact that $\phi$ is continuous, and only changes its slope finitely many times on each closed edge. 
For part (b), note that if $p, q \in S$ then $\phi(p)=\phi(q)$ whereas if $p \in S$ but $q \in S_\epsilon \backslash S$ (for a sufficiently small $\epsilon$), then the outgoing slope of $\phi$ from $p$ to $q$ is at least $1$. Therefore $-\div(\phi)(p) \geq \val_S(p)$.
\end{proof}
It is known that there is a unique break divisor representative in any equivalence class of divisors in degree $g$ (\cite{MZ, ABKS}). Here we give a new proof of this fact, which is better suited for the study of semibreak divisors.
\begin{Proposition}\label{prop:uniq}
Let $\Gamma$ be a metric graph of genus $g$. If $D, D' \in \Div^g(\Gamma)$ are two distinct break divisors then $D\not\sim D'$.
\end{Proposition}
\begin{proof}
	Suppose, for a contradiction, that there exist two distinct linearly equivalent break divisors $D$ and $D'$. Then $D'=D+\div(\phi)$ for some $\phi\in R(\Gamma)$. Let $S$ be the closed admissible subset of $\Gamma$ where $\phi$ attains its minimum, and let $C$ be a sufficiently small cut determined by $S$ such that 
	\begin{equation} \label{eq:cutchoice}
	C \cap \supp(D)=\emptyset \quad , \quad  C\cap\supp(\div(\phi))=\emptyset \, .
	\end{equation}
		
	Let $S'=S\cup C$. Then $S$ is a deformation retract of the admissible open set $S'$, and in particular $p_a(S')=p_a(S)$. By \eqref{eq:cutchoice} and Proposition \ref{prop:break_char} we obtain 
\[
\deg(D'|_S)=\deg(D'|_{S'})\geq p_a(S')=p_a(S) \, .
\]
By the definition of $S$ and Lemma~\ref{lem:maxprin} we have
\[
-\deg(\div(\phi)|_S) \geq c(S) \, .
\]
Together with Lemma~\ref{lem:cut}, we conclude:
\[
\deg(D|_S)=\deg(D'|_S)-\deg(\div(\phi)|_S)\geq p_a(S)+\psi(S)-p_a(S)+1\geq \psi(S)+1 \, ,
\]
This implies, by Proposition~\ref{prop:break_char} (iii), that $D$ cannot be a break divisor, which is a contradiction.
\end{proof}

The above argument does not guarantee the uniqueness of semibreak representatives even for degree $g-1$. Notice that indeed, an effective divisor class can have more than one semibreak divisor (see Figure \ref{fig:equivalent}). 
\begin{figure}[H]
\begin{tikzpicture}[scale=.75]
\draw[black, -](0,0)--(1,0);
\draw(-1,0) circle (1);
\draw(1.7,0) circle (.7);
\fill[black] (0,0) circle (.1);
\fill[black] (1,0) circle (.1);
\node at (-.3,0){$u$};
\node at (1.3,0){$v$};
\end{tikzpicture}
\caption{\label{fig:equivalent} Linearly equivalent semibreak divisors: $(u) \sim (v)$.}
\end{figure}
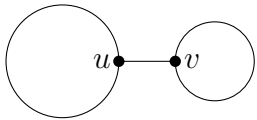 

By a slight modification of our proof of Proposition~\ref{prop:uniq}, we obtain the following sufficient condition for the uniqueness of semibreak divisors in their equivalence classes. 

\begin{Proposition}\label{prop:distinctsemibreak}
Let $\Gamma$ be a metric graph of genus $g$. Fix $0 \leq d \leq g$. Assume for each cut $C$ of $\Gamma$ we have $\size(C) \geq d+1$.
If $D, D' \in \Div^d(\Gamma)$ are two distinct semibreak divisors then $D\not\sim D'$.
\end{Proposition}
\begin{proof}
	Suppose, for the sake of contradiction, that there exist two linearly equivalent semibreak divisors $D$ and $D'$ of degree $d$. Then $D'=D+\div(\phi)$ for some $\phi\in R(\Gamma)$. Let $S$ be the closed admissible subset of $\Gamma$ where $\phi$ attains its minimum. By the definition of $S$ and Lemma~\ref{lem:maxprin} we have
\[
-\deg(\div(\phi)|_S) \geq c(S) \, .
\]
	As $D'$ is a semibreak divisor, it is effective, thus 
	\[\deg(D|_S)+\deg(\div(\phi)|_S) = \deg(D'|_S)\geq 0\, .\]
	 Consequently, 
	\[ d \geq \deg(D|_S) \geq -\deg(\div(\phi)|_S) \geq c(S) \geq d+1 \, ,\] 
	which is a contradiction.
\end{proof}

\begin{Remark}
Proposition~\ref{prop:distinctsemibreak} should be compared with \cite[Theorem 1.8]{BN07}. In fact, in the proof, we only use the fact that $D$ and $D'$ are effective. In other words, we are precisely proving the metric graph analogue of \cite[Theorem 1.8]{BN07}.
\end{Remark}

The following result will be useful later, and its proof is a slight modification of the previous proof.
\begin{Lemma}\label{lem:semibreak_equiv_cond}
Let $\Gamma$ be a metric graph, and let $D$ and $D'$ be distinct effective divisors on $\Gamma$ with $D\sim D'$. Then there exists a closed admissible subset $S$ of $\Gamma$ such that for every $p\in S$ we have $D(p)\geq \val_S(p)$.
\end{Lemma}
\begin{proof}
By assumption, we have $D'=D+\div(\phi)$ for some $\phi\in R(\Gamma)$. Let $S$ be the closed admissible subset of $\Gamma$ where $\phi$ attains its minimum. By the definition of $S$ and Lemma~\ref{lem:maxprin}, for any $p \in S$ we have
\[
-\div(\phi)(p) \geq \val_S(p) \, .
\]
As $D'$ is effective, it follows that
	\[D(p)+\div(\phi)(p) = D'(p)\geq 0\, .\]
	 Consequently, 
	\[ D(p) \geq -\div(\phi)(p) \geq \val_S(p) \, .\] 

\end{proof}

\subsection{Integral semibreak divisors}

Let $G = (V(G), E(G))$ be a finite (unweighted) graph. Let $\Gamma$ be a metric graph of genus $g$, obtained from $G$ by turning each edge to an edge segment of length $1$. Let $\Div^d(G)$ denote those elements of $\Div^d(\Gamma)$ that are supported on $V(G)$. We will refer to such divisors as {\em integral}. 
Let $|D|_G$ denote the set of all effective integral divisors linearly equivalent to $D$. It is known that, for an integral divisor $D$, we have $|D|_G \ne \emptyset$ if and only if $|D| \ne \emptyset$ \cite{HKN,luoye}.

The following result implies that there is an entirely integral version of Theorem~\ref{thm:repr_with_below_break}. 
\begin{Proposition}\label{prop:integral}
Let $D \in \Div^d(G)$, with $0 \leq d \leq g$, and assume $|D|_G\neq \emptyset$. Then there exists an integral semibreak divisor in $|D|$. Moreover, each semibreak divisor in $|D|$ is integral. 

\end{Proposition}
\begin{proof}
The existence part follows from the proof of Theorem \ref{thm:repr_with_below_break}: we can choose $E$ such that $\supp(E)\subseteq V(G)$. Then $D+E$ is still integral. To obtain our semibreak divisor, we successively apply steps (1) or (2). It is enough to show that in the above case these both give integral break divisors. For step (1), this is trivial. For step (2), notice that for an integral divisor $D+E$, the set $\partial\minmax\left(D+E\right)$ is contained in $\supp(D+E)$, and hence is contained in $V(G)$. It follows that  $\epsilon = \dist\left(S,V_{\Gamma} \backslash S\right)$ is an integer and hence that $D'$ and $E'$ will still be integral divisors, proving the existence part of the statement.

Let us suppose for a contradiction that there also exists a non-integral break divisor $D''$ in $|D|$. This means that there exists $p\notin V(G)$ such that $D''(p)>0$, which implies that $D''(p)=1$, as $D''$ is a semibreak divisor. Let us suppose that the two vertices of the edge of $p$ are $u$ and $v$.
Take some $\phi\in R(\Gamma)$ such that $D''=D'+\div(\phi)$. Without loss of generality we can assume that $\phi(u)=k$ is an integer. Then $D''(p)=1$ implies that the slope of the segment between $u$ and $p$ and the slope of the segment between $p$ and $v$ differ by one. Let us suppose that the slope of the segment between $u$ and $p$ is $t\in \mathbb{Z}$, and let $\dist(u,p)=a$ (which is not an integer). Then $\phi(p)=k+at$ and $\phi(v)=k+at+(1-a)(t-1)=k+t-1+a$. Hence $\phi(v)$ is not an integer. Let $S=\{w\in V(G): \phi(w)\in\mathbb{Z}\}$. Then $S\subsetneq V(G)$. Let $C$ be the cut determined by $S$, i.e., the set of edges of $G$ where exactly one endpoint is from $S$. As $S$ is a proper subset of $V(G)$ and $G$ is connected, $C$ is nonempty. If $\div(\phi)$ is constant zero on an edge of $G$, then the value of $\phi$ on the two endpoints differ by an integer, since the slopes are integer and the length of each edge is one. Hence each edge of the cut $C$ needs to have an interior point $q$ where $\div(\phi)(q)\neq 0$. As $D'$ is integral, this means that $\div(\phi)(q)> 0$ on each of these points, and thus $D''$ has positive number of chips in the interior of each edge of $C$, contradicting the fact that $D''$ is a semibreak divisor.
\end{proof}

\begin{Remark}\label{rmk:partialorien}
It follows from Proposition~\ref{prop:integral} that there is an entirely finite graph version of the theory of semibreak divisors. For example, any effective divisor of degree at most the genus on a finite graph $G$ is linearly equivalent to some semibreak divisor.
One could directly use a combinatorial analogue of our constructions to prove this (slightly weaker) result. Alternatively, one can give a completely combinatorial proof using the theory of orientations on graphs as described in \cite{ABKS} and \cite{backman} (see loc. cit. for the notation and terminology): let $D\in\Div^d_+(G)$ with $0\leq d\leq g$. Pick a point $q \in \Gamma$. By \cite[Theorem 5.7]{backman} we have  $D - (q) \sim D_{\mathcal{O}}$ for some `$q$-connected partial orientation' $\mathcal{O}$. For any unoriented edge in $\mathcal{O}$, pick an arbitrary orientation to obtain a (full) $q$-connected orientation $\mathcal{O}'$. By construction $D_{\mathcal{O}} \leq D_{\mathcal{O}'}$. Let $E = D_{\mathcal{O}'} + (q)$. This is a break divisor by \cite[Lemma 3.3]{ABKS}. Let $D' = D_{\mathcal{O}} + (q)$ which is dominated by $E$.
\end{Remark}

\section{Computational aspects}\label{sec:compute}

In this section we show that for an effective divisor, we can find a linearly equivalent semibreak divisor in polynomial time if the input data can be given by rational numbers.

Suppose that the metric graph $\Gamma$ has rational edge lengths. We call a divisor $D$ rational, if for all $p$ with $D(p)\neq 0$, the distance of $p$ from each branch point is rational.

We will encode a rational number $a/b$ by the pair $(a,b)$, where $a$ and $b$ are integers, but they need not be relatively prime.
We encode a metric graph in the following way: For each branch point we write down the list of edges incident to it, along with the edge lengths. We encode a point of an edge by writing down which edge it is on, and what is its distance from one of the endpoints. We encode a divisor $D$ by writing down $D(p)$ along with the encoding of $p$ for each point $p\in\Gamma$ such that $D(p)\neq 0$. We only work with effective divisors of degree at most $g$, hence we can suppose that $D(p)\leq g$ for each $p$. We might need more space for encoding a divisor than for encoding the graph if the distance of some $p$ with $D(p)\neq 0$ has a large denominator. However, as the numbers $D(p)$ are at most $g$, this is the only factor that can make the code of a divisor large.
We will need to encode closed convex sets. Let the spanning set of a closed convex set $S$ be the following:
$$\spset(S)=\bigcup_{e\in E_\Gamma}\{\partial(e\cap S)\},$$
where $\partial\emptyset=\emptyset$.
Then $\spset(S)$ is a finite set of points, and $S=\conv(\spset(S))$. We encode $S$ by giving the points of $\spset(S)$.

\begin{Remark}
The encoding of a metric graph outlined above has polynomial size in the genus plus the number of bits needed for writing down the edge lengths.
A more succinct encoding of metric graphs would be to only write down how many edges of each length connect two given vertices. However, to write down divisors we need to be able to distinguish between edges connecting the same pair of branch points, hence we need the first type of encoding.

In the case of discrete graphs (i.e., in the model where edge lengths are one and we are not allowed to place chips on the interior points of edges), the model of only encoding edge multiplicities is more natural. Using this more succinct encoding, computational problems are potentially more difficult. In particular, in the integral setting, deciding whether $|D|\neq \emptyset$ can be done in polynomial time in the less succinct model (\cite{BS13}), but it is open whether it can be done in polynomial time in the succinct model.
\end{Remark}

\begin{Theorem} \label{thm:compute}
If $\Gamma$ is a metric graph with rational edge lengths, and $D$ is an effective rational divisor, then a semibreak divisor linearly equivalent to $D$ can be found in polynomial time.
\end{Theorem}
\begin{proof}
The trivial case of $\Gamma$ being a circle can once again be excluded.

We need to be able to do the procedure in Theorem \ref{thm:repr_with_below_break} algorithmically:
at the first step, we can choose $E=(g-d)\cdot p$ where $d=\deg(D)$ and $p$ is an arbitrary branch point of $\Gamma$. Then $D+E$ is still rational, and the largest denominator in the encoding did not grow.

The next step is to find $\ME(D+E)$ and $S=\minmax(D+E)$. We address this issue later. If we have $\minmax(D+E)$, we have to decide whether $\supp(E)\cap \partial S\neq \emptyset$. This can be done since $\partial S$ has at most $|E_\Gamma|$ points. If $\supp(E)\cap \partial S\neq \emptyset$, the operations of case (1) can be trivially done in polynomial time and the resulting $D$ and $E$ are rational. Moreover, the largest denominator in the encoding does not grow. If $\supp(E)\cap \partial S= \emptyset$, then we need to find $\epsilon = \dist(S,V_\Gamma\setminus S)$. For this, we need to check distances along polynomially many edges. Note that $\partial S$ is contained in $\supp (D+E)$, hence all the distances between $\partial S$ and $V_\Gamma \setminus S$ are rational. Thus, $\epsilon$ is also rational, and so are the updates $D'$ and $E'$. The effect of the firing can also be computed in polynomial time, and as we only add or subtract distances, the largest denominator in the encoding of $D+E$ does not grow.

We need to update the divisors $D$ and $E$ polynomially many times (i.e., the loop in the proof of Theorem \ref{thm:repr_with_below_break} is executed polynomially many times): after any update, $\ME(D+E)$ does not increase, and if it does not decrease, then the number of branch points in $\minmax(D+E)$ increases. At the beginning, $\ME(D+E)$ is at most $g$, as $\deg(D|_S)\leq \deg(D)\leq g$, and $\psi(S)\geq 0$. Hence there are at most $g \cdot |V_\Gamma|$ updates.

It is left to show that one can find $\ME(D+E)$ and $\minmax(D+E)$ in polynomial time. For any divisor $D+E$ and convex set $S$, $\Error(D+E,S)$ can be computed in polynomial time, hence it is enough to find $\minmax(D+E)$, and then $\ME(D+E)=\Error(D+E,\minmax(D+E))$.

For finding $\minmax(D+E)$, we will use submodular minimization.
For a finite set $A$, a set function $f:2^{A}\to \mathbb{R}$ is called submodular if 
\[f(X\cap Y)+f(X\cup Y)\leq f(X)+f(Y),\]
 for all subsets $X$ and $Y$ of $A$. It is known \cite{Schrijver_submod, IFF_submod} (see also \cite[Chapter 45]{Schrijver_book}) that, if $f$ takes rational values and for any set $X\subseteq A$ the value of $f(X)$ can be computed in polynomial time, then a set minimizing $f$ can be found in polynomial time. Schrijver and Iwata--Fleischer--Fujishige give combinatorial strongly polynomial algorithms \cite{Schrijver_submod, IFF_submod} that achieve this goal.

By Lemma \ref{lem:Err} (a), for two closed convex sets $S_1$ and $S_2$, 
\[\Error(D+E,S_1\cap S_2) + \Error(D+E,\conv(S_1\cup S_2))\geq \Error(D+E,S_1) + \Error(D+E,S_2).\]
 This essentially means that $(-1)\cdot\Error(D+E,.)$ is a submodular set function, and we need to find a smallest minimizing set for it. Though $\Error$ is defined for infinitely many sets, we show how to turn the problem into a finite setting, and then we can apply a submodular minimization algorithm.

Let us take a more refined model $G$ of $\Gamma$, where $V(G)=V_\Gamma \cup \supp(D+E)$. Then $|V(G)|$ is still polynomial in the input size.
By Proposition \ref{prop:minmax}, $\minmax(D+E)$ is a convex set which is the union of vertices and closed edges of $G$, hence it is enough to look for $\minmax(D+E)$ among these sets. We can introduce $\conv_G$, as the convex hull with respect to the model $G$, i.e., $\conv_G(S)$ is the union of $S$ and all closed edge segments of the model $G$ where both endpoints are from $S$.

For a set $S$ which is the union of vertices and closed edges of $G$, let $\spset_G(S)=S\cap V(G)$. Notice that if $S$ is convex in $G$, and it is the union of vertices and closed edges of $G$, then $S=\conv_G(\spset_G(S))$.
Also, if $S_1\subseteq S_2$, then $\conv_G(S_1)\subseteq \conv_G(S_2)$.

Let us define the set function 
\[
f\colon 2^{V(G)}\to \mathbb{R},\;\; X\mapsto (-1)\cdot\Error(D+E,\conv_G(X)) \, .
\]
By the argument above, $\minmax(D+E)=\conv(X)$ for the smallest $f$-minimizing set $X$. We claim that $f$ is submodular. This follows from the variant of Lemma \ref{l:comparison with convex hull}, where we replace $\conv$ with $\conv_G$ (the proof is completely analogous) and the elementary fact that 
\begin{align*}
\conv_G(X\cup Y)&=\conv_G(\conv_G(X)\cup \conv_G(Y)) \quad \text{, and}\\
\conv_G(X\cap Y)&=\conv_G(X)\cap \conv_G(Y)
\end{align*}
for all subsets $X$ and $Y$ of $V(G)$.

A submodular minimization algorithm only gives us a minimizing set, and not necessarily a smallest one. To solve this problem, take $f'(X)=f(X)+\frac{|X|}{2|V(G)|}$. Then only the fractional part of $f'(X)$ depends on the additional term, hence a set minimizing $f'$ is a minimizing set for $f$, and an inclusion-minimal (and hence smallest by Proposition \ref{prop:minmax}) among those. It is easy to check that $f'$ is also submodular, and the values of $f'$ are still computable in polynomial time.
\end{proof}

\section{Generic effective divisor classes}
\label{sec:generic}

\subsection{Effective loci}
Let $\Gamma$ be a metric graph of genus $g$. 
Recall, for $D \in \Div(\Gamma)$, its {\em rank} $r(D)$ is defined by the properties that $r(D) = -1$ if $|D| = \emptyset$, and $r(D) \geq s \geq 0$ if for all $E \in \Div_+^s(\Gamma)$ we have $|D-E| \ne \emptyset$. Clearly, the rank of a divisor $D$ only depends on its linear equivalence class $[D]$.

The {\em tropical Riemann-Roch theorem} of \cite{GK, MZ} (see also \cite{BN07}) states that 
\[r(D) - r(K-D) = \deg(D) - g + 1\, ,\] 
where $K = \sum_{p \in \Gamma}{(\val(p)-2)(p)}$.

The \emph{tropical Jacobian} of $\Gamma$ can be defined as the $g$-dimensional real torus
\[
\Jac(\Gamma) = H_1(\Gamma, \RR) / H_1(\Gamma, \ZZ) \, .
\]
For each choice of a base point $q\in \Gamma$, there is a natural, continuous map $\Phi_q\colon\Gamma\to \Jac(\Gamma)$ sending $q$ to $0$, the \emph{Abel-Jacobi map} (cf.\ \cite{MZ, baker2011metric}). This map is piecewise linear, which means that locally on $\Gamma$ it factors through a piecewise linear map to the vector space $H_1(\Gamma,\RR)$. As $\Jac(\Gamma)$ is a topological group, we may use its addition and the map $\Phi_q$ to define maps 
\[
\Phi_q^{(d)}\colon \Div^d(\Gamma)\to \Jac(\Gamma) 
\]
for $d\geq 0$. Of course, the composite $\Gamma^d\to \Div^d(\Gamma)\to \Jac(\Gamma)$ will still be piecewise linear. As $\Gamma^d$ is compact and $\Jac(\Gamma)$ is Hausdorff, it follows from {\em closed map lemma} (see e.g. \cite[Lemma A.52]{Lee}) that $\Gamma^d \to \Jac(\Gamma)$ is a closed map. In particular, the \emph{effective locus} $\widetilde W_d= \Phi_q^{(d)}(\Div^d_+(\Gamma))$ is a closed polyhedral subset of $\Jac(\Gamma)$. It follows from the tropical Riemann-Roch theorem that, for $d \geq g$,  we have $\widetilde W_d = \Jac(\Gamma)$.

If we denote
\[
\Pic^d(\Gamma) = \Div^d(\Gamma) / \Prin(\Gamma)\, ,
\] 
then it is the content of the tropical Abel-Jacobi theorem that $\Phi_q^{(d)}$ factors through the natural map $S^{(d)} \colon \Div^d(\Gamma) \rightarrow \Pic^d(\Gamma)$, and that the induced morphism $\Pic^d(\Gamma)\to \Jac(\Gamma)$ is a bijection (\cite{MZ}). We endow  $\Pic^d(\Gamma)$ the topology inherited from this bijection. Under this bijection, the effective locus $\widetilde W_d$ corresponds to the locus of {\em effective divisors classes} $W_d$, i.e. those divisors classes $[D]\in \Pic^d(\Gamma)$ in degree $d$ such that $|D| \ne \emptyset$.

\subsection{Generic semibreak divisors in effective loci}
We are interested in {\em generic} properties of $W_d$, i.e. properties that hold on a dense open subset of $W_d$.

Let $0\leq d\leq g$. The set $\Div^d_+(\Gamma)$ is endowed with the quotient topology coming from its identification with $\Gamma^d$ modulo the action of the symmetric group $\mathfrak{S}_d$. 

Let $\SB_d\subseteq \Div^d_+(\Gamma)$ denote the set of all semibreak divisors of degree $d$. Its preimage in $\Gamma^d$ is the union of all sets of the form $\overline e_1\times \cdots \times\overline e_d$, where $e_1,\ldots, e_g$ are distinct open edges of $\Gamma$ such that $\Gamma\backslash \bigcup _{i=1}^g e_i$ is connected. Therefore, the set $\SB_d$ is {\em closed} in $\Div^d_+(\Gamma)$. 
Let $\Sigma_d$ denote the interior of $\SB_d$. The preimage of $\Sigma_d$ in $\Gamma^d$ is the union of all sets of the form $e_1\times \cdots \times e_d$, with $e_1,\ldots, e_g$ as above. In particular, $\Sigma_d$ is {\em open} in $\SB_d$.

\begin{Lemma} \label{lem:inj}
For any $D \in \Sigma_d$, we have $|D| = \{D\}$ and $r(D) = 0$.
\end{Lemma}
\begin{proof}
Let $S$ be a closed admissible subset of $\Gamma$, and let $C$ be a cut determined by $S$. Because $\Gamma\backslash C$ is disconnected, there exists a component $e$ of $C$ such that $D(p)=0$ for all $p\in \overline e$. In particular, for the unique point $p\in\overline e \cap S$ we have $D(p)<\val_S(p)$. The statement $|D| = \{D\}$ now follows from Lemma~\ref{lem:semibreak_equiv_cond}.
If $q \not\in \supp(D)$, it follows from $|D| =\{D\}$ that we must have $|D-(q)| = \emptyset$ and therefore $r(D) = 0$.
\end{proof}

\begin{Remark}
One can alternatively show, using the burning algorithm (see e.g. \cite{luoye}), that any $D \in \Sigma_d$ is $D$ is {\em universally reduced} (i.e. $q$-reduced for all $q \in \Gamma$) which is equivalent to having $|D| = \{D\}$ (see \cite[Lemma~4.19]{ABKS}).
\end{Remark}

\begin{Theorem}
\label{thm:birational}
Let $\Gamma$ be a metric graph of genus $g$, and fix $0 \leq d \leq g$.  
\begin{itemize}
\item[(a)] The tropical Abel-Jacobi map $S^{(d)} \colon \Div^d_+(\Gamma) \rightarrow \Pic^d(\Gamma)$ is `birational' onto its image. More precisely, there exists an open dense subset $U_d \subseteq  W_d$  such that the induced map $(S^{(d)})^{-1}(U_d)\to U_d$ is a homeomorphism. 
\item[(b)] $W_d \subseteq \Pic^d(\Gamma)$ is of pure dimension $d$. 
\end{itemize}
\end{Theorem}

\begin{proof}
(a) Let $U_d=S^{(d)}(\Sigma_d)$. It follows from Lemma \ref{lem:inj} that $(S^{(d)})^{-1}(U_d)=\Sigma_d$.
By Theorem \ref{thm:repr_with_below_break}  the induced map 
\[\SB_d\to W_d\] 
is surjective. It is also a closed map (by the {\em closed map lemma}), because $\SB_d$ is compact and $W_d$ is Hausdorff. In particular, the topology on $ W_d$ coincides with the quotient topology. 

As $\Sigma_d$ is dense in $\SB_d$, it follows from the closedness of the map that $U_d=S^{(d)}(\Sigma_d)$ is dense in $ W_d$. It is a direct consequence of Lemma \ref{lem:inj} that $\Sigma_d=(S^{(d)})^{-1}(U_d)$ and $\Sigma_d\to U_d$ is a bijection. Since $\Sigma_d$ is open in $\SB_d$, this implies that that $U_d$ is open as well. As closedness is local on the target, the induced map $\Sigma_d\to U_d$ is closed again. And as a continuous and closed bijection it must be a homeomorphism. 

For part (b), note that $\Sigma_d$, and hence $U_d$, are purely $d$-dimensional; every component of $\Sigma_d$ can be identified with a $d$-dimensional open polyhedron in $\Gamma^d$ under the quotient map $\Gamma^d\to \Div^d_+(\Gamma)$. It follows immediately that the closure $ W_d$ of $U_d$ is also purely $d$-dimensional.
\end{proof}

\begin{Remark} \label{rmk:tropvsag1}
\begin{itemize}
\item[]
\item[(i)] 
As $W_d$ is the $d$-fold sum of $W_1$, which is easily seen to be purely $1$-dimensional, it follows directly from the subadditivity of dimensions of sums that $W_d$ is at most $d$-dimensional. With the additional ingredient that $W_g=\Pic^g(\Gamma)$ is $g$-dimensional it even follows that the dimension of $W_d$ is equal to $d$ (cf.\ \cite[Proposition 3.6]{LPP}). Note that this argument does \emph{not} immediately imply that $W_d$ is purely $d$-dimensional, as sums of pure-dimensional polyhedral sets are not pure-dimensional in general. However, with some extra care this approach will yield a different proof of the pure-dimensionality, and with some more work even of Theorem \ref{thm:repr_with_below_break}. But unlike our proof, this approach does not yield an algorithm that calculates a semibreak divisor that is linear equivalent to a given effective divisor.

\item[(ii)] Another way to prove the pure-dimensionality of $W_d$ is by tropicalization. It is well-known that there exists a Mumford curve $C$ whose Berkovich analytification has $\Gamma$ as its skeleton. Combining the results \cite[Theorem 1.3]{BR} and \cite[Theorem 6.9]{Gubler} then yields the statement. Of course, this approach is highly non-constructive.
\end{itemize}
\end{Remark}

\begin{Theorem}
\label{thm:generic rank} 
Let $\Gamma$ be a metric graph of genus $g$ and let $d$ be a nonnegative integer. Then there exists an open dense subset $ U_d \subseteq W_d$ of the effective locus such that, for $[D]\in U_d$, we have
\[
r(D) = 
\begin{cases}
d-g &\text{ if } d > g, \\
0 &\text{ if } 0 \leq d \leq g.
\end{cases}
\]
\end{Theorem}

\begin{proof}
The case $d > g$ is an elementary consequence of the tropical Riemann-Roch Theorem:
\begin{itemize}
\item If $d\geq 2g-1$ we can take 
\[U_d= W_d=\Pic^d(\Gamma)\, .\] 
Since $\deg(K) = 2g-2$, for every degree $d$ divisor $D$ we have $r(K-D)=-1$, hence $r(D)=d-g$ by the tropical Riemann-Roch. 
\item If $g < d \leq 2g-2$ we can take 
\[
U_d= \Pic^d(\Gamma) -  \left(S^{(2g-2)}(K)- W_{2g-2-d} \right) \, .\] 
We claim $U_d$ is a dense open subset of $ W_d  =\Pic^d(\Gamma)$. 
This follows from the fact that $W_{2g-2-d}$ is a closed polyhedral subset of $\Pic^d(\Gamma)$, of dimension {\em at most} $2g-2-d <g$. This certainly follows from Theorem~\ref{thm:birational} (b), but it is more elementary and follows directly from definitions. 

If $[D] \in U_d$ then, by definition of $U_d$, $K-D$ is not equivalent to an effective divisor.  Therefore, $r(K-D)=-1$ and  $r(D)=d-g$ by the tropical Riemann-Roch Theorem.
\end{itemize}

Assume $d\leq g$. In this case, we may take $U_d=S^{(d)}(\Sigma_d)$ which is open and dense in $ W_d$ by Theorem \ref{thm:birational}. 
If $[D] \in U_d$, then $D$ is linear equivalent to some $D'\in \Sigma_d$. It follows from Lemma \ref{lem:inj} that $r(D')=0$, finishing the proof.
\end{proof}

\begin{Remark} \label{rmk:tropvsag2}
Unlike in algebraic geometry, a property that holds generically on $W_d$ does not automatically hold generically for $\Div_+^d(\Gamma)$. This is because the tropical Abel-Jacobi map $S^{(d)}$ may contract facets of $\Div^d_+(\Gamma)$. For example, if $\Gamma$ is a chain of two loops, as depicted in Figure \ref{fig:equivalent}, and $e$ is the {\em bridge} (i.e the edge  connecting the two circles) then all divisors of the form $(p)+(q)$ with $p,q\in e$ are linear equivalent and of rank $1$. Of course, the set of these divisors has nonempty interior and hence there does not exist a dense open subset of $\Div^2_+(\Gamma)$ where the rank is $0$. On the other hand, by Theorem \ref{thm:generic rank}, there exists a dense open subset $U_2 \subseteq W_2$ such that $r(D)=0$ whenever $[D]\in U_2$.
\end{Remark}

\input{Semibreak.bbl}

%\bibliography{Refs}
%\bibliographystyle{alpha}

\end{document}

%% file: Semibreak.bbl

%% file: Semibreak.bbl
\begin{bibdiv}
\begin{biblist}

\bib{ABKS}{article}{
      author={An, Yang},
      author={Baker, Matthew},
      author={Kuperberg, Greg},
      author={Shokrieh, Farbod},
       title={Canonical representatives for divisor classes on tropical curves
  and the matrix-tree theorem},
        date={2014},
        ISSN={2050-5094},
     journal={Forum Math. Sigma},
      volume={2},
       pages={e24, 25},
         url={https://doi-org.proxy.library.cornell.edu/10.1017/fms.2014.25},
      review={\MR{3264262}},
}

\bib{backman}{article}{
      author={Backman, Spencer},
       title={Riemann-{R}och theory for graph orientations},
        date={2017},
        ISSN={0001-8708},
     journal={Adv. Math.},
      volume={309},
       pages={655\ndash 691},
  url={https://doi-org.proxy.library.cornell.edu/10.1016/j.aim.2017.01.005},
      review={\MR{3607288}},
}

\bib{baker}{article}{
      author={Baker, Matthew},
       title={Specialization of linear systems from curves to graphs},
        date={2008},
        ISSN={1937-0652},
     journal={Algebra Number Theory},
      volume={2},
      number={6},
       pages={613\ndash 653},
  url={https://doi-org.proxy.library.cornell.edu/10.2140/ant.2008.2.613},
        note={With an appendix by Brian Conrad},
      review={\MR{2448666}},
}

\bib{Berkovich}{book}{
      author={Berkovich, Vladimir~G.},
       title={Spectral theory and analytic geometry over non-{A}rchimedean
  fields},
      series={Mathematical Surveys and Monographs},
   publisher={American Mathematical Society, Providence, RI},
        date={1990},
      volume={33},
        ISBN={0-8218-1534-2},
      review={\MR{1070709}},
}

\bib{baker2011metric}{article}{
      author={Baker, Matthew},
      author={Faber, Xander},
       title={Metric properties of the tropical {A}bel-{J}acobi map},
        date={2011},
        ISSN={0925-9899},
     journal={J. Algebraic Combin.},
      volume={33},
      number={3},
       pages={349\ndash 381},
         url={http://dx.doi.org/10.1007/s10801-010-0247-3},
      review={\MR{2772537}},
}

\bib{BJ}{incollection}{
      author={Baker, Matthew},
      author={Jensen, David},
       title={Degeneration of linear series from the tropical point of view and
  applications},
        date={2016},
   booktitle={Nonarchimedean and tropical geometry},
      series={Simons Symp.},
   publisher={Springer, [Cham]},
       pages={365\ndash 433},
      review={\MR{3702316}},
}

\bib{BN07}{article}{
      author={Baker, Matthew},
      author={Norine, Serguei},
       title={Riemann-{R}och and {A}bel-{J}acobi theory on a finite graph},
        date={2007},
        ISSN={0001-8708},
     journal={Adv. Math.},
      volume={215},
      number={2},
       pages={766\ndash 788},
  url={https://doi-org.proxy.library.cornell.edu/10.1016/j.aim.2007.04.012},
      review={\MR{2355607}},
}

\bib{BPR}{article}{
      author={Baker, Matthew},
      author={Payne, Sam},
      author={Rabinoff, Joseph},
       title={Nonarchimedean geometry, tropicalization, and metrics on curves},
        date={2016},
        ISSN={2214-2584},
     journal={Algebr. Geom.},
      volume={3},
      number={1},
       pages={63\ndash 105},
         url={https://doi-org.proxy.library.cornell.edu/10.14231/AG-2016-004},
      review={\MR{3455421}},
}

\bib{BR}{article}{
      author={Baker, Matthew},
      author={Rabinoff, Joseph},
       title={The skeleton of the {J}acobian, the {J}acobian of the skeleton,
  and lifting meromorphic functions from tropical to algebraic curves},
        date={2015},
        ISSN={1073-7928},
     journal={Int. Math. Res. Not. IMRN},
      number={16},
       pages={7436\ndash 7472},
         url={https://doi-org.proxy.library.cornell.edu/10.1093/imrn/rnu168},
      review={\MR{3428970}},
}

\bib{BS13}{article}{
      author={Baker, Matthew},
      author={Shokrieh, Farbod},
       title={Chip-firing games, potential theory on graphs, and spanning
  trees},
        date={2013},
        ISSN={0097-3165},
     journal={J. Combin. Theory Ser. A},
      volume={120},
      number={1},
       pages={164\ndash 182},
  url={https://doi-org.proxy.library.cornell.edu/10.1016/j.jcta.2012.07.011},
      review={\MR{2971705}},
}

\bib{CDPR}{article}{
      author={Cools, Filip},
      author={Draisma, Jan},
      author={Payne, Sam},
      author={Robeva, Elina},
       title={A tropical proof of the {B}rill-{N}oether theorem},
        date={2012},
        ISSN={0001-8708},
     journal={Adv. Math.},
      volume={230},
      number={2},
       pages={759\ndash 776},
  url={https://doi-org.proxy.library.cornell.edu/10.1016/j.aim.2012.02.019},
      review={\MR{2914965}},
}

\bib{CJP}{article}{
      author={Cartwright, Dustin},
      author={Jensen, David},
      author={Payne, Sam},
       title={Lifting divisors on a generic chain of loops},
        date={2015},
        ISSN={0008-4395},
     journal={Canad. Math. Bull.},
      volume={58},
      number={2},
       pages={250\ndash 262},
  url={https://doi-org.proxy.library.cornell.edu/10.4153/CMB-2014-050-2},
      review={\MR{3334919}},
}

\bib{GriffHarr}{book}{
      author={Griffiths, Phillip},
      author={Harris, Joseph},
       title={Principles of algebraic geometry},
      series={Wiley Classics Library},
   publisher={John Wiley \& Sons, Inc., New York},
        date={1994},
        ISBN={0-471-05059-8},
         url={https://doi-org.proxy.library.cornell.edu/10.1002/9781118032527},
        note={Reprint of the 1978 original},
      review={\MR{1288523}},
}

\bib{GK}{article}{
      author={Gathmann, Andreas},
      author={Kerber, Michael},
       title={A {R}iemann-{R}och theorem in tropical geometry},
        date={2008},
        ISSN={0025-5874},
     journal={Math. Z.},
      volume={259},
      number={1},
       pages={217\ndash 230},
  url={https://doi-org.proxy.library.cornell.edu/10.1007/s00209-007-0222-4},
      review={\MR{2377750}},
}

\bib{Gubler}{article}{
      author={Gubler, Walter},
       title={Tropical varieties for non-{A}rchimedean analytic spaces},
        date={2007},
        ISSN={0020-9910},
     journal={Invent. Math.},
      volume={169},
      number={2},
       pages={321\ndash 376},
  url={https://doi-org.proxy.library.cornell.edu/10.1007/s00222-007-0048-z},
      review={\MR{2318559}},
}

\bib{HKN}{article}{
      author={Hladk\'y, Jan},
      author={Kr\'al, Daniel},
      author={Norine, Serguei},
       title={Rank of divisors on tropical curves},
        date={2013},
        ISSN={0097-3165},
     journal={J. Combin. Theory Ser. A},
      volume={120},
      number={7},
       pages={1521\ndash 1538},
  url={https://doi-org.proxy.library.cornell.edu/10.1016/j.jcta.2013.05.002},
      review={\MR{3092681}},
}

\bib{IFF_submod}{article}{
      author={Iwata, Satoru},
      author={Fleischer, Lisa},
      author={Fujishige, Satoru},
       title={A combinatorial strongly polynomial algorithm for minimizing
  submodular functions},
        date={2001},
        ISSN={0004-5411},
     journal={J. ACM},
      volume={48},
      number={4},
       pages={761\ndash 777},
         url={https://doi-org.proxy.library.cornell.edu/10.1145/502090.502096},
      review={\MR{2144929}},
}

\bib{JP1}{article}{
      author={Jensen, David},
      author={Payne, Sam},
       title={Tropical independence {I}: {S}hapes of divisors and a proof of
  the {G}ieseker-{P}etri theorem},
        date={2014},
        ISSN={1937-0652},
     journal={Algebra Number Theory},
      volume={8},
      number={9},
       pages={2043\ndash 2066},
  url={https://doi-org.proxy.library.cornell.edu/10.2140/ant.2014.8.2043},
      review={\MR{3294386}},
}

\bib{JP2}{article}{
      author={Jensen, David},
      author={Payne, Sam},
       title={Tropical independence {II}: {T}he maximal rank conjecture for
  quadrics},
        date={2016},
        ISSN={1937-0652},
     journal={Algebra Number Theory},
      volume={10},
      number={8},
       pages={1601\ndash 1640},
  url={https://doi-org.proxy.library.cornell.edu/10.2140/ant.2016.10.1601},
      review={\MR{3556794}},
}

\bib{hiperTutte}{article}{
      author={K\'alm\'an, Tam\'as},
       title={A version of {T}utte's polynomial for hypergraphs},
        date={2013},
        ISSN={0001-8708},
     journal={Adv. Math.},
      volume={244},
       pages={823\ndash 873},
  url={https://doi-org.proxy.library.cornell.edu/10.1016/j.aim.2013.06.001},
      review={\MR{3077890}},
}

\bib{Lee}{book}{
      author={Lee, John~M.},
       title={Introduction to smooth manifolds},
     edition={Second},
      series={Graduate Texts in Mathematics},
   publisher={Springer, New York},
        date={2013},
      volume={218},
        ISBN={978-1-4419-9981-8},
      review={\MR{2954043}},
}

\bib{LPP}{article}{
      author={Lim, Chang~Mou},
      author={Payne, Sam},
      author={Potashnik, Natasha},
       title={A note on {B}rill-{N}oether theory and rank-determining sets for
  metric graphs},
        date={2012},
        ISSN={1073-7928},
     journal={Int. Math. Res. Not. IMRN},
      number={23},
       pages={5484\ndash 5504},
         url={https://doi-org.proxy.library.cornell.edu/10.1093/imrn/rnr233},
      review={\MR{2999150}},
}

\bib{luoye}{article}{
      author={Luo, Ye},
       title={Rank-determining sets of metric graphs},
        date={2011},
        ISSN={0097-3165},
     journal={J. Combin. Theory Ser. A},
      volume={118},
      number={6},
       pages={1775\ndash 1793},
  url={https://doi-org.proxy.library.cornell.edu/10.1016/j.jcta.2011.03.002},
      review={\MR{2793609}},
}

\bib{MZ}{incollection}{
      author={Mikhalkin, Grigory},
      author={Zharkov, Ilia},
       title={Tropical curves, their {J}acobians and theta functions},
        date={2008},
   booktitle={Curves and abelian varieties},
      series={Contemp. Math.},
      volume={465},
   publisher={Amer. Math. Soc., Providence, RI},
       pages={203\ndash 230},
  url={https://doi-org.proxy.library.cornell.edu/10.1090/conm/465/09104},
      review={\MR{2457739}},
}

\bib{Schrijver_submod}{article}{
      author={Schrijver, Alexander},
       title={A combinatorial algorithm minimizing submodular functions in
  strongly polynomial time},
        date={2000},
        ISSN={0095-8956},
     journal={J. Combin. Theory Ser. B},
      volume={80},
      number={2},
       pages={346\ndash 355},
  url={https://doi-org.proxy.library.cornell.edu/10.1006/jctb.2000.1989},
      review={\MR{1794698}},
}

\bib{Schrijver_book}{book}{
      author={Schrijver, Alexander},
       title={Combinatorial optimization. {P}olyhedra and efficiency. {V}ol.
  {B}},
      series={Algorithms and Combinatorics},
   publisher={Springer-Verlag, Berlin},
        date={2003},
      volume={24},
        ISBN={3-540-44389-4},
        note={Matroids, trees, stable sets, Chapters 39--69},
      review={\MR{1956925}},
}

\end{biblist}
\end{bibdiv}
